\documentclass[a4paper]{article}

\usepackage{fullpage}
\usepackage[T1]{fontenc}
\usepackage{amsmath}
\usepackage{amssymb}
\usepackage{amsthm}
\usepackage{dsfont}
\usepackage{upgreek}
\usepackage{graphicx}
\usepackage{enumitem}
\usepackage[round,comma,authoryear]{natbib}
\usepackage{float}
\usepackage{subfig}
\usepackage{url}
\usepackage[pdftex,dvipsnames]{xcolor}
\usepackage[colorlinks,citecolor=blue,urlcolor=blue]{hyperref}
\usepackage[linecolor=OliveGreen,backgroundcolor=OliveGreen!25,bordercolor=OliveGreen,textsize=small]{todonotes}
\usepackage{diagbox}

\theoremstyle{plain}
\newtheorem{theorem}{Theorem}
\newtheorem{lemma}[theorem]{Lemma}
\newtheorem{proposition}[theorem]{Proposition}
\newtheorem{corollary}[theorem]{Corollary}

\newtheorem{condition}{Condition}

\newtheorem*{example*}{Example}

\newtheorem{remark}{Remark}
\newtheorem*{remark*}{Remark}
\newtheorem{acknowledgement*}{Acknowledgement}

\newcommand{\diag}{\operatorname{diag}}

\newcommand{\argmin}{\operatorname*{\arg\min}}
\newcommand{\argmax}{\operatorname*{\arg\max}}
\newcommand{\sgn}{\operatorname{sgn}}

\newcommand{\M}{\mathcal{M}}
\renewcommand{\O}{\mathcal{O}}
\newcommand{\N}{\mathbb{N}}
\newcommand{\R}{\mathbb{R}}
\newcommand{\Rquer}{\overline{\R}}
\newcommand{\ind}{\mathds{1}}
\newcommand{\eps}{\varepsilon}
\renewcommand{\L}{{\text{\tiny L}}}
\newcommand{\LS}{{\text{\tiny LS}}}
\newcommand{\betaL}{\hat\beta_\L}
\newcommand{\betaLS}{\hat\beta_\LS}
\newcommand{\uLS}{\hat u_\LS}
\newcommand{\epsLS}{\hat \eps_\LS}
\renewcommand{\th}{\text{\tiny th}}
\newcommand{\Int}{\text{\tiny int}}

\newcommand{\longto}{\longrightarrow}
\newcommand{\dto}{\overset{d}{\longrightarrow}}
\newcommand{\pto}{\overset{p}{\longrightarrow}}

\renewcommand{\emptyset}{\varnothing}

\newcommand{\pd}[2]{\frac{\partial #1}{\partial #2}}

\graphicspath{{../plots/}}

\begin{document}

\title{Uniformly Valid Confidence Sets Based on the Lasso}

\author{Karl Ewald and Ulrike Schneider\\[2ex]
Vienna University of Technology} 

\date{}

\maketitle

\begin{abstract}
In a linear regression model of fixed dimension $p \leq n$, we construct
confidence regions for the unknown parameter vector based on the Lasso
estimator that uniformly and exactly hold the prescribed in finite samples
as well as in an asymptotic setup. We thereby quantify estimation
uncertainty as well as the ``post-model selection error'' of this
estimator. More concretely, in finite samples with Gaussian errors and
asymptotically in the case where the Lasso estimator is tuned to perform
conservative model selection, we derive exact formulas for computing the
minimal coverage probability over the entire parameter space for a large
class of shapes for the confidence sets, thus enabling the construction of
valid confidence regions based on the Lasso estimator in these settings.
The choice of shape for the confidence sets and comparison with the
confidence ellipse based on the least-squares estimator is also discussed.
Moreover, in the case where the Lasso estimator is tuned to enable
consistent model selection, we give a simple confidence region with minimal
coverage probability converging to one. Finally, we also treat the case of
unknown error variance and present some ideas for extensions.
\end{abstract}

\section{Introduction} \label{sec:intro}

The Lasso estimator as introduced in \cite{Tibshirani96} as well as many
variants thereof have gained strong interest in the statistics community
and in applied areas over the past two decades. As is well known, the main
attraction of the Lasso estimator lies in its ability to perform model
selection and parameter estimation at very low computational cost, see for
instance \cite{AllineyRuzinsky94}, \cite{EfronEtAl04} and
\cite{RossetZhu07}, and in the fact that the estimator can be used in
high-dimensional settings where the number of variables $p$ exceeds the
number of observations $n$ (``$p \gg n$'').


Recent years have seen an increased interest on how to perform valid
inference in connection with these types of estimators.
\cite{PoetscherSchneider10} construct valid confidence intervals based on
the Lasso as well as related estimators in the framework of linear
regression models with orthogonal design and give an in-depth analysis of
the problems and challenges that arise in this context. Generalizations of
these results to a moderate-dimensional (orthogonal) setting where $p \leq
n$ but $p$ diverging with $n$ can be found in \cite{Schneider16}.


In a general high-dimensional setting with $p \gg n$, confidence regions
and confidence intervals in connection with the Lasso estimator have
recently been treated by different approaches. Based on
\cite{ZhangZhang14}, several papers including \cite{VdGeerEtAl14},
\cite{JavanmardMontanari14}, \cite{CanerKock18} and \cite{VdGeerStucky16}
use the idea of ``de-sparsifying'' the Lasso estimator. In case where $p
\leq n$ this approach essentially reduces to using the least-squares (LS)
estimator for inference. In that sense this theory leaves a gap on how to
construct confidence regions based on the Lasso estimator in a
low-dimensional framework to provide uncertainty quantification for the
Lasso estimator in this case.


\cite{LeeEtAl16} consider finite-sample results for confidence intervals
in connection with the Lasso estimator yet these authors take a different
route in that their intervals are not set to cover the true parameter, but
a pseudo-true value that depends on the selected model and coincides with
the true parameter if the selected model is correct. All inference is
conditional on the selected model. A model-dependent parameter is also
covered in \cite{BerkEtAl13} who discuss an intricate procedure for
obtaining confidence regions for a pseudo-true parameter in connection with
arbitrary model selection procedures. 


\emph{In this paper, we construct confidence sets based on the Lasso
estimator for the entire unknown parameter vector. We stress that while in
the low-dimensional case the LS estimator can be employed to build
confidence regions, the Lasso estimator is still used in such a framework,
naturally entailing the question on how to conduct valid inference, and our
results also quantify the worst-case estimation (``post-model selection'')
error of this method. Moreover, \cite{SchneiderEwald17TR} show that in high
dimensions, the Lasso estimator may in fact act as a low-dimensional
procedure in which case the results of this paper can also be applied.}


One of the challenges of this task lies in the well-known fact that the
finite-sample distribution of the Lasso estimator depends on the unknown
parameter in a complicated manner. This phenomenon does not vanish for
large samples as can be seen within a so-called moving-asymptotic framework
(see \cite{PoetscherLeeb09} for a detailed analysis in orthogonal design)
and also occurs for related estimators. 
In order to construct valid confidence sets, we need to know the smallest
coverage probability occurring over the whole parameter space.
\cite{PoetscherSchneider10} derive a formula for the minimal coverage
probability of fixed-width confidence intervals based on the Lasso
estimator in one dimension using knowledge of its finite-sample
distribution. In the general case, this finite-sample distribution is not
known, so it is not clear how to obtain an expression for the coverage
probability in more than one dimension. Additionally, this coverage
probability clearly depends on the shape that is used for the confidence
set and it is a not clear a priori what this shape should be. We do the
following.


While the finite sample distribution and therefore the coverage probability
for any kind of set based on the Lasso estimator is unknown in general
dimensions, we show that computing the \emph{minimal} coverage probability
can actually be carried out without this explicit knowledge. We obtain an
explicit formula for the minimal coverage probability by, in a way,
deferring the minimization problem into the objective function that defines
the estimator, as is depicted in Section~\ref{sec:finitesample}. For the
confidence regions, we consider a large class of shapes that is determined
by a condition involving the regressor matrix. This class encompasses the
elliptic shape one would use if the confidence region was based on the LS
estimator, thus enabling comparisons with the LS confidence ellipse.
Analogously to the fixed-width intervals in \cite{PoetscherSchneider10},
the confidence regions we consider are random only through their centering
at the Lasso estimator (which is also in line with the setup in the
literature for high-dimensional settings, see for instance
\citealt{VdGeerEtAl14}). Asymptotically, we distinguish between two regimes
for the tuning parameters which we call conservative and consistent tuning.
As suggested from the results in \cite{PoetscherSchneider10}, our results
from finite samples essentially carry over asymptotically when the
estimator is tuned conservatively. In the case of consistent tuning, the
uniform convergence rate of the estimator is slower than $n^{-1/2}$ and we
give the asymptotic distribution of the Lasso estimator when scaled by the
appropriate factor corresponding to the uniform convergence rate, as well
as suggesting a simple construction for a confidence set in that case.


The remaining paper is organized as follows. In Section~\ref{sec:setting}
we set the framework by stating the model, defining the estimator and
introducing some notation. The main result giving the formula for the
minimal coverage probability is presented in Section~\ref{sec:finitesample}
and subsequently Section~\ref{sec:construct} is devoted to discussing how
to concretely construct the corresponding confidence sets, as well as their
relationship to the confidence ellipse based on the LS estimator. We treat
the case of unknown error variance in Section~\ref{sec:extensions}, as well
as several ideas for extensions and further considerations. In
Section~\ref{sec:asymp} we derive asymptotic results both for the case of
conservative and the case of consistent model selection.
Section~\ref{sec:conclusion} concludes. All proofs are deferred to
Appendix~\ref{sec:proofs}.

Literature on distributional properties of the Lasso estimator in the
low-dimensional setting ($p \leq n$) include the often-cited paper by
\cite{KnightFu00} who derive the asymptotic distribution when the estimator
is tuned to perform conservative model selection. \cite{PoetscherLeeb09}
give a detailed analysis in the framework of a linear regression model with
orthogonal design and derive the distribution of the Lasso estimator in
finite samples as well as in the two asymptotic regimes of consistent and
conservative tuning. Implications of these results for confidence intervals
are analyzed in \cite{PoetscherSchneider10} and generalizations to a
moderate-dimensional setting where $p \leq n$ but $p$ diverging with $n$
are contained in \cite{PoetscherSchneider11} and \cite{Schneider16}.

\section{Setting and assumptions} \label{sec:setting}

Consider the linear model 
$$
y = X \beta + \eps,
$$
where $y$ is the observed $n \times 1$ data vector, X the $n \times p$
regressor matrix which is assumed to be non-stochastic with full column
rank $p$, $\beta \in \R^p$ is the true parameter vector and $\eps$ the
unobserved error term defined on some probability space
$(\Omega,\mathcal{A},P)$ and consisting of independent and identically
distributed components with mean 0 and finite variance $\sigma^2$. We
consider a componentwise tuned Lasso estimator $\betaL$, defined as the
unique solution to the minimization problem
$$
\min_{\upbeta \in \R^p} L_n(\upbeta) = \min_{\upbeta \in \R^p}  
\|y - X\upbeta\|^2 + 2 \sum_{j = 1}^p \lambda_{n,j} |\upbeta_j|,
$$
where $\lambda_{n,j}$, are non-negative and non-random componentwise tuning
parameters that allow to exclude parameters from penalization. Note that if
$\lambda_{n,j} = 0$ for all $j$, this estimator is equal to the ordinary
least-squares (LS) estimator $\betaLS$ and that $\lambda_{n,j} = c > 0$ for
all $j$ corresponds to the ``classical'' Lasso estimator as proposed by
\cite{Tibshirani96}. For later use, let $\lambda_n =
(\lambda_{n,1},\dots,\lambda_{n,p})'$ and $\Lambda_n = \diag(\lambda_n)$,
the diagonal matrix whose diagonal elements are given by the components of
$\lambda_n$. We use $\ind_{\{.\}}$ for the indicator function and make the
following obvious definitions. For $a \in \R^p$ and $B \subseteq \R^p$, the
set $a + B = B + a \subseteq \R^p$ is defined as the set $\{a + b : b \in
B\}$. For a $p \times p$ matrix $\bar C$ and a scalar $c$, the sets $\bar C
B$ and $cB$ in $\R^p$ are  $\{\bar Cb : b \in B\} \subseteq \R^p$ and $\{cb
: b \in B\} \subseteq \R^p$, respectively. Finally, for $k \in \N$, $I_k$
stands for the $k \times k $ identity matrix and $\Rquer$ denotes the
extended real line $\R \cup \{-\infty,\infty\}$.

\section{Finite-sample results} \label{sec:finitesample}

We aim to construct confidence sets for the entire parameter vector $\beta$
based on the Lasso estimator $\betaL$. That means that for a non-random set
$M \subseteq \R^p$, we consider sets of the form
$$
\betaL - M = \{\betaL - m : m \in M\},
$$
which have to satisfy that the probability of actually covering the unknown
parameter $\beta$ never (for no value of $\beta$) falls below a prescribed
level $1 - \alpha$ with $\alpha \in [0,1]$. In other words, we need
$P_\beta(\beta \in \betaL - M) \geq 1 - \alpha$ for all $\beta \in \R^p$
(where we stress the dependence of the probability measure on $\beta$
whenever it occurs), so that
$$
\inf_{\beta \in \R^p} P_\beta(\beta \in \betaL - M) \geq 1 - \alpha.
$$
In order to achieve this, we need to be able to compute this ``infimal''
(minimal) coverage probability. \emph{Throughout this and the two
subsequent sections} we suppose that the errors as normally distributed
\begin{equation*}
\eps \sim N(0, \sigma^2 I_n),
\end{equation*}
although our results do not heavily depend on this assumption, also see
Remark~\ref{rem:normassump}. The assumption that will be removed for
asymptotic results in Section~\ref{sec:asymp}. We will show that the
minimum occurs when the components of the unknown parameter become large in
absolute value by essentially doing the following. We reparametrize the
objective function defining the Lasso estimator so that the dependence on
the unknown parameter becomes more transparent and easier to handle. We
then consider the limiting cases of the objective functions when the
components of the unknown parameter vector $\beta$ become large in absolute
value (that is, tend to $+\infty$ or $-\infty$). We will see that it is
possible to minimize the resulting objective functions explicitly, with
minimizers that follow a shifted normal distribution that has the same
covariance matrix as the LS estimator and by construction do not depend on
the unknown parameter. Finally, we will show that the infimal coverage
probability of the proposed sets is indeed ``achieved'' for one of these
finitely many limiting cases.


To state the main theorem, we need several definitions. First we define the
reparame\-trized objective function $Q_n(u) = L_n(\beta + n^{-1/2}u) -
L_n(\beta)$ so that $Q_n$ is uniquely minimized at $\hat u_n =
n^{1/2}(\betaL - \beta)$, the estimation error scaled by $n^{1/2}$. Of
course, this scaling factor is arbitrary in finite samples, but proves to
be of advantage when considering the problem in large samples in
Section~\ref{subsec:conservative}. We can write $Q_n$ as
$$
Q_n(u) = u'C_n u - 2 u'W_n + 2n^{-1/2}\sum_{j=1}^p
\lambda_{n,j} \left[|u_j + n^{1/2}\beta_j| - |n^{1/2}\beta_j|\right],
$$
where $C_n = X'X/n$ and $W_n =  n^{-1/2}X'\eps \sim N(0,\sigma^2 C_n)$.
Note that for a set $M \subseteq \R^p$ we then have
$$
P_\beta(\beta \in \betaL - n^{-1/2}M) = P_\beta(\hat u_n \in M).
$$
The above mentioned limiting cases of the objective function that we
consider are defined as
\begin{equation} \label{eq:Qnd}
Q_n^d(u) = u' C_n u - 2u'W_n + 2n^{-1/2}\sum_{j=1}^p \lambda_{n,j} d_j u_j,
\end{equation}
where $d = (d_1,\dots,d_p)' \in \{-1,1\}^p$. Holding $W_n$ fixed for a
moment, we indeed see that 
$$
Q^d_n(u) = \lim_{d_j\beta_j \to \infty \atop j=1,\dots,p} Q_n(u).
$$
As shorthand notation, we write $\hat u_n^d$ for the unique minimizer of
$Q_n^d$. To define the shape that we want to consider for the confidence
regions, we introduce the following notation. For $m \in \R^p$, a vector $d
\in \{-1,1\}^p$ and a matrix $\bar C \in \R^{p \times p}$, we define
$$
A_{\bar C}^d(m) = \bigcap_{j=1}^p
\{z \in \R^p: d_j(\bar Cm)_j \leq d_j(\bar Cz)_j, d_j z_j \leq 0\}. 
$$
The set $A_{\bar C}^d(m)$ is an intersection of $2p$ half-spaces, $p$ of
which determine the orthant the set is located in via the parameter $d$.
The other $p$ half-spaces are defined by hyperplanes that intersect at the
point $m$. Figure~\ref{fig:AdSet} shows one example of such a set. Note
that in general, $A_{\bar C}^d(m)$ could be non-empty also for $\sgn(m)
\neq -d$.
\begin{figure}
\centering \includegraphics[width=0.48\textwidth]{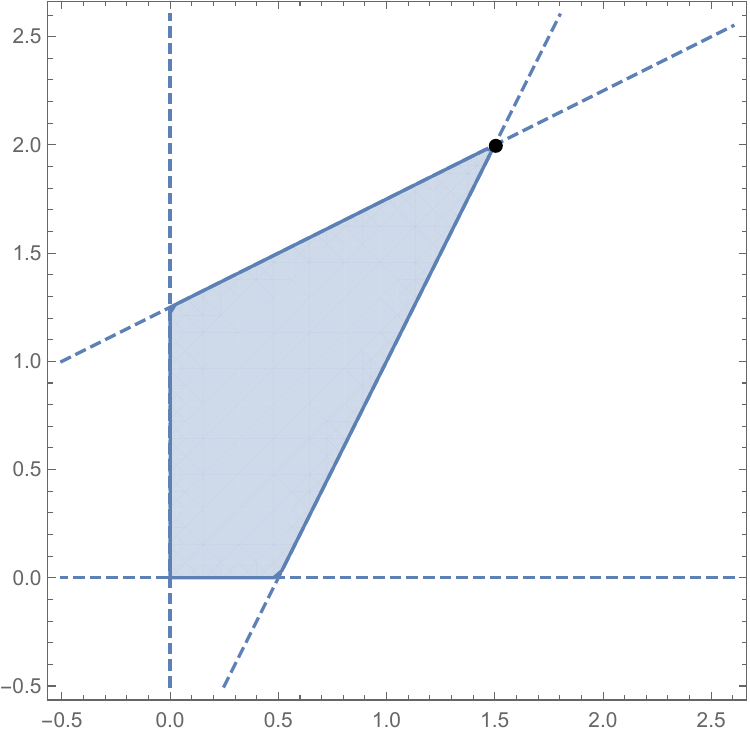}
\caption{\label{fig:AdSet} The set $A_{\bar C}^{-\iota}(m)$ with
$\iota = (1,1)'$, $m =
(1.5,2)'$ and $\bar C = \left(\protect\begin{smallmatrix} 1 & -0.5 \\ -0.5 & 1 \protect\end{smallmatrix}\right)$ 
along with the hyperplanes defining the set. The point $m = (1.5,2)'$ is
displayed as a dot.}
\end{figure} 
The sets we consider are determined by the following condition.

\begin{condition} \label{cond:A} 
Let $\bar C \in \R^{p\times p}$ be given. We say that a set $M \subseteq
\R^p$ satisfies Condition~\ref{cond:A} with matrix $\bar C$ if
$$
A_{\bar C}^d(m) \subseteq M
$$
for all $d \in \{-1,1\}^p$ and for all $m \in M$.
\end{condition}

The above condition will be discussed in more detail in
Section~\ref{sec:construct}. Using this notation, we can now state the main
theorem.

\begin{theorem}\label{ther:infprob}
If $M_n \subseteq \R^p$ is non-random and satisfies Condition~\ref{cond:A}
with $\bar C = C_n$, then
$$
\inf_{\beta \in \R^p} P_\beta(\hat u_n \in M_n) =
\min_{d \in \{-1,1\}^p} P(\hat u_n^d \in M_n),
$$
where $\hat u_n^d \sim N(-n^{-1/2}C_n^{-1}\Lambda_n d,
\sigma^2C_n^{-1})$.
\end{theorem}

The distributions of $\hat u^d_n$ determining the formula for the infimal
coverage probability are shifted normal distributions with the same
covariance matrix as the corresponding (shifted and scaled) LS
estimator $\uLS = n^{1/2}(\betaLS - \beta)$ and mean that depends on the
regressors and the vector of tuning parameters. 

\begin{remark} \label{rem:normassump}
Note that (the proof of) Theorem~\ref{ther:infprob} does not hinge on the
normality assumption, as it exploits the structure of the underlying
optimization problem rather than stochastic properties of the error
distribution. Different error distributions could be used in
Theorem~\ref{ther:infprob}, only the distributions of $\hat u_n^d$ would
have to be adapted accordingly.
\end{remark}

Since Condition~\ref{cond:A} for $p=1$ simply requires the corresponding
set $M_n$ to be an interval containing zero, Theorem~\ref{ther:infprob} is
indeed a generalization of the formula in Theorem 5(a) in
\cite{PoetscherSchneider10}, as discussed in the introduction. (To make the
connection, note that the tuning parameter $\eta_n$ in that reference
corresponds to a component $n^{-1/2}\lambda_{n,j}$ of the vector of tuning
parameters in our paper.) The following obvious corollary specifies the
resulting valid confidence region based on the Lasso estimator.

\begin{corollary} \label{cor:mincov}
Let $0 < \alpha < 1$. If $M_n \subseteq \R^p$ is non-random and satisfies
Condition~\ref{cond:A} with $\bar C = C_n$, as well as $\min_{d \in
\{-1,1\}^p} P(\hat u_n^d \in M_n) = 1 - \alpha$ with $\hat u_n^d \sim
N(-n^{-1/2}C_n^{-1}\Lambda_n d,\sigma^2C_n^{-1})$, then
$$
\inf_{\beta \in \R^p} P_\beta(\beta \in \betaL - n^{-1/2}M_n) = 1 - \alpha.
$$
\end{corollary}

\section{Constructing the confidence set} \label{sec:construct}

We now turn to discussing the important matter of how to choose an
appropriate set $M_n \subseteq \R^p$ for some desired level of confidence
$1 - \alpha$ by discussing concrete shapes for the confidence regions as
well as their size and relation to confidence sets based on the LS
estimator. As mentioned in the previous section, we need to find a set $M_n
\subseteq \R^p$ that satisfies Condition~\ref{cond:A} with $ \bar C= C_n$
and such that $\min_{d \in \{-1,1\}^p} P(\hat u_n^d \in M_n) = 1 - \alpha$
where
$$
\hat u_n^d \sim N(-n^{-1/2}C_n^{-1}\Lambda_n d, \sigma^2 C_n^{-1}).
$$
The resulting confidence set for $\beta$ is then the scaled and shifted set
$\betaL - M_n/n^{1/2}$. If we would base the set on the LS estimator
$\betaLS$ instead of $\betaL$, the canonical and best choice for $M_n$ in
terms of volume is an ellipse determined by the contour lines of a
$N(0,\sigma^2C_n^{-1})$-distribution, the \emph{$C_n$-ellipse}. Given the
fact that the covariance matrix of the distributions of $\hat
u^d_n$ is in fact $\sigma^2C_n^{-1}$, in addition to the fact that the
means of the distributions average to 0, it is reasonable to consider the
$C_n$-ellipse as a shape in connection with the Lasso estimator also. As
stated in the following proposition, this shape complies with
Condition~\ref{cond:A}.

\begin{proposition} \label{prop:Cellipse}
The $C_n$-ellipse given by
$$
E_{C_n}(k)=\{z \in \R^p: z'C_n z \leq k\}
$$
satisfies Condition~\ref{cond:A} with $\bar C = C_n$ for any $k > 0$.
\end{proposition} 

How to choose the parameter $k$ for a given level of coverage $1 - \alpha$
is stated in the next proposition.

\begin{proposition} \label{prop:dstar}
For any $k > 0$, we have that
$$
\argmin_{d \in \{-1,1\}^p} P\left(\hat u_n^d \in E_{C_n}(k)\right) = 
\argmax_{d \in \{-1,1\}^p} \|C_n^{-1/2}\Lambda_n d \|.
$$
\end{proposition} 
Note that if $d^*$ solves the above optimization problem, so does $-d^*$.
To finally obtain the confidence ellipse based on the Lasso estimator, pick
any such optimizer $d^*$ and compute $k^*>0$ so that $P(u_n^{d^*} \in
E_{C_n}(k^*)) = 1 - \alpha$, which is easily done based on the following
proposition.

\begin{proposition} \label{prop:kstar}
For $0 < \alpha < 1$ we have $P(\hat u_n^d \in E_{C_n}(\sigma^2 \kappa)) =
1 - \alpha$ for
$$
\kappa = (\chi^2_{p,\nu})^{-1}(1-\alpha),
$$
where $(\chi^2_{p,\nu})^{-1}$ is the quantile function of a non-central
$\chi^2$-distribution with $p$ degrees of freedom and non-centrality parameter
$$ 
\nu = \frac{1}{n\sigma^2} d'\Lambda_nC_n^{-1}\Lambda_nd.
$$
\end{proposition} 

Note that Proposition~\ref{prop:kstar} also shows that the ellipse
$E_{C_n}(k^*)$, and therefore the resulting confidence set based on the
Lasso estimator, is larger in volume than the one based on the LS
estimator, since $P(\beta \in \betaLS - E_{C_n}(\sigma^2\kappa)) = 1 -
\alpha$ is satisfied for $\kappa = (\chi^2_p)^{-1}(1-\alpha)$ where
$(\chi^2_p)^{-1}$ is the quantile function of a (central)
$\chi^2$-distribution with $p$ degrees of freedom. Clearly, the difference
in size will increase as the tuning parameters become large as then the
non-centrality parameter $\nu$ will grow. These observations are in line
with the findings in \cite{PoetscherSchneider10} who show that a confidence
interval based on the Lasso estimator is larger than a confidence interval
based on the LS estimator with the same coverage probability.

When comparing the two confidence sets, we emphasize that since the
ellipses are centered at different values, the smaller ellipse based
on the LS estimator is in general \emph{not} contained in the ellipse
based on the Lasso estimator. This, as well as the difference in
volume between the two ellipses, will also be illustrated in the
example below.


It is quite obvious that the $C_n$-ellipse is not optimal as a shape for
confidence sets based on the Lasso estimator since we can get higher
coverage with a set of the same volume by adjusting the ellipse ``towards''
the contour lines of the $N(-n^{-1/2}C_n^{-1}\Lambda_n
d^*,\sigma^2C_n^{-1})$-distributions (in such a way that
Condition~\ref{cond:A} is preserved). To find the best shape possible, one
would have to minimize the volume of the set over all possible shapes
satisfying Condition~\ref{cond:A} subject to the constraint of holding the
prescribed minimal coverage probability. This is a highly complex
optimization problem and we do not dwell further on this subject here, but
illustrate possible ways to construct ``good'' sets, as shown in the
example below. Before discussing this further, note that the following
proposition shows that it is easy to find the closure of an arbitrary
subset of $\R^p$ with respect to Condition~\ref{cond:A}.

\begin{proposition} \label{prop:closeA} 
For any $M \subseteq \R^p$, the set
$$
\bigcup_{m \in M} \bigcup_{d \in \{-1,1\}^p} A^d_{\bar C}(m)
$$
is the smallest set containing $M$ that satisfies
Condition~\ref{cond:A}.
\end{proposition}

We now provide an example for $p = 2$ illustrating the difference between
the confidence ellipse based on the LS estimator and the one based on the
Lasso, as well as how to choose a better shape in terms of volume for the
confidence set based on the Lasso estimator. The simulations and
calculations were carried out using the statistical software package {\tt
R}. The example is set up in the following way. We let $n=20$ and generate
the $(n \times 2)$-matrix $X$ using independent and identically distributed
standard normal entries that are transformed row-wise by an appropriate $(2
\times 2)$-matrix in order to get
$$
C_n = \frac{X'X}{n} = \begin{pmatrix}
1 & -0.5\\
-0.5 & 1 
\end{pmatrix}.
$$
We generate the data vector $y$ from the corresponding linear model with
$\sigma^2 = 1$ (so that $\eps \sim N(0,I_n)$) and true parameter chosen as
$\beta = (1,0)'$. We compute the Lasso estimator using the {\tt
glmnet}-package and tuning parameters $\lambda_{n,1} = \lambda_{n,2} =
\sqrt{n}/2$ (asymptotically corresponding to what we will refer to as
\emph{conservative model selection} in the subsequent section). We also
considered estimators where the tuning parameters were chosen by 10-fold
cross-validation (as provided in the {\tt glmnet}-package) which ended up
yielding comparable results for the estimator.


We then constructed confidence ellipses with level $\alpha = 0.05$ based on
both the LS and the Lasso estimator in the manner described earlier in this
section. The resulting sets are shown in Figure~\ref{fig:LassoVsLS}.
\begin{figure}
\centering
\includegraphics[width=0.48\textwidth]{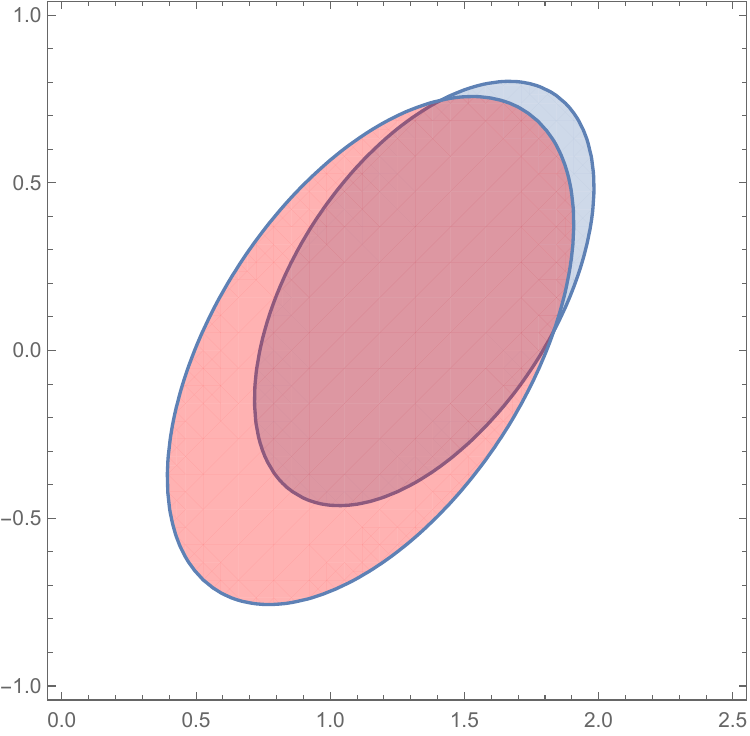} 
\caption{\label{fig:LassoVsLS}The confidence ellipses based on and
centered at the Lasso estimator $\betaL = (1.15, 0)'$ (red) and the
smaller one based on and centered at the LS estimator $\betaLS =
(1.35, 0.17)'$ (blue), respectively.} 
\end{figure}
The plot clearly illustrates the above described fact that the confidence
ellipse based on the Lasso estimator is larger than the confidence ellipse
that is based on the LS estimator. Also, the two sets are overlapping by a
large amount (in fact, the maximal distance between the two estimators is
controlled by Proposition~\ref{prop:lsdifflasso} in the Appendix). However,
the LS ellipse is not entirely contained in the one based on the Lasso,
stressing the fact the Theorem~\ref{ther:infprob} yields non-trivial sets.


The above comparison between the two ellipses, however, is somewhat unfair
in the sense that the shape used for both confidence sets is the optimal
one (in terms of volume) for the LS estimator, but, as discussed above, not
for the Lasso estimator. With the optimal shape for a Lasso confidence set
being unknown, we at least want to find a shape that improves upon the
ellipse. As a basis for this, we consider the union of the contour sets
corresponding to the distributions of $\hat u_n^d$, that is, the $2^p$
shifted $C_n$-ellipses
$$
U_n(k) \; = \bigcup_{d \in \{-1,1\} ^p} E_{C_n}(k) - 
n^{-1/2} C_n^{-1} \Lambda_n d,
$$ 
where each set in the union is of optimal shape for the corresponding
distribution of $\hat u_n^d$. As a starting point, we choose $k$ so that
$P(\hat u_n^d \in E_{C_n}(k) - n^{-1/2} C_n^{-1} \Lambda_n d) = 1 - \alpha$
(note that $k$ is then simply the parameter of the $C_n$-ellipse used for
the LS estimator, but any $k > 0$ such that $U_n(k)$ satisfies $P(\hat
u_n^d \in U_n(k)) \geq 1 - \alpha$ works). Clearly, this set is still too
large and will not satisfy Condition~\ref{cond:A}, so we need to address
these two issues. First, we add all points necessary so that the resulting
set satisfies Condition~\ref{cond:A}. Proposition~\ref{prop:closeA} ensures
that
$$
\bigcup_{m \in U_k} \bigcup_{d \in \{-1,1\}^p} A_{C_n}^d(m)
$$
fulfills the desired condition. Note that in this particular case, it is
fairly straightforward to see that this set is simply given by the convex
hull of the shifted ellipses $U_n(k)$. Finally, to get the smallest set
with this shape that still holds the prescribed level of coverage, we
iteratively adjust the set by reducing the parameter $k$ and re-calculate
the minimal coverage probability of the resulting set until the desired
minimal coverage probability is reached (up to an arbitrary level of
precision). The resulting alternatively shaped set is depicted in
Figure~\ref{fig:lassoCIimproved}, \subref{subfig:construct} showing the
midpoints of the $2^p = 4$ ellipses used in the construction and
\subref{subfig:improved} displaying the new confidence set on top of the
elliptic confidence region based on the Lasso as devised before. It is
obvious that the new shape has slightly less volume than the ellipse.
\begin{figure}
\centering
\subfloat[]{\includegraphics[width=0.48\textwidth]{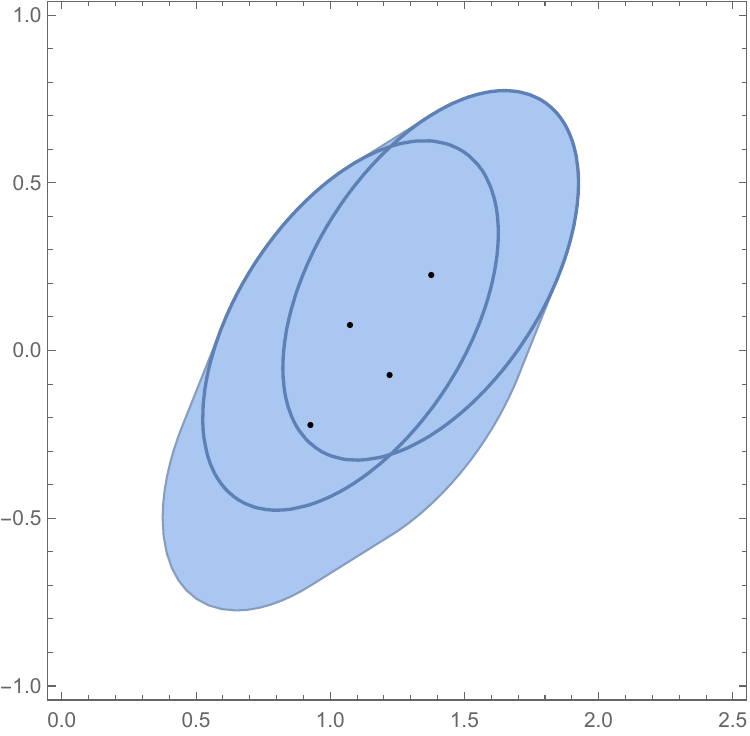}
\label{subfig:construct}}
\subfloat[]{\includegraphics[width=0.48\textwidth]{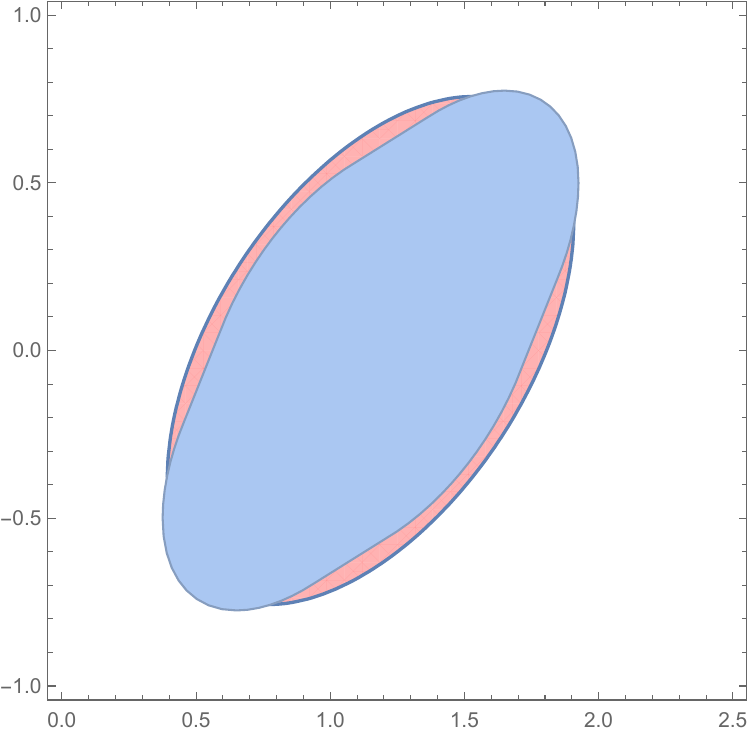}
\label{subfig:improved}}
\caption{\label{fig:lassoCIimproved} \protect\subref{subfig:construct}
Construction of the alternative shape based on $2^p = 4$ ellipses with
their centers displayed as dots. \protect\subref{subfig:improved} The
resulting improved confidence set with the alternative shape (blue)
and the previous elliptic shape (red), both based on at the Lasso
estimator $\betaL = (1.15,0)'$.}
\end{figure}

\section{Extensions and further considerations} \label{sec:extensions}

In Section~\ref{subsec:unknownvar} we extend the previous results for the
case of unknown error variance. Furthermore, we provide some insights on
how the coverage probability of Lasso confidence regions might vary over
the parameter space in Section~\ref{subsec:paramspace} and illustrate some
ideas on how to build confidence intervals for a single component of the
parameter vector in Section~\ref{subsec:paramspace}, the latter two
sections considering for the simple case of $p=2$.

\subsection{Unknown error variance} \label{subsec:unknownvar}

As the results in Section~\ref{sec:construct} on how to construct the
confidence regions use knowledge of the error variance $\sigma^2$. We
now turn to the more realistic setting when the error variance is
unknown and extend our findings to this framework. Let
$$
\hat\sigma^2 = \frac{1}{n-k} \epsLS'\epsLS,
$$
the usual unbiased estimator of $\sigma^2$ based on the LS residuals
$\epsLS = y - X\betaLS$. 

To apply the previous results to this setting, we let the tuning parameter
$\lambda$ depend on the variance estimate in the following way. For this
subsection, set $\lambda_n = \gamma_n/\hat\sigma$, where $\gamma_n \in
\R^p$ with $\gamma_{n,j} \geq 0$ let $\hat u_n^d$ be defined as before.
Since the main argument for proving the results leading up to
Corollary~\ref{cor:mincov} depend on the minimization problem rather than
on stochastic properties, inspection of the corresponding proofs reveals
that the minimal coverage probability can still be computed
correspondingly. Not too surprisingly, rather than using (non-central)
normal distributions, we need to consider (non-central)
$t$-distributions\footnote{A $p$-dimensional multivariate
$T(k,\mu,\Sigma)$ with $k$ degrees of freedom, non-centrality parameter
$\mu \in \R^p$ and positive definite matrix $\Sigma \in \R^{p \times p}$
has Lebesgue density function
$$
f(t) = 
\frac{\Gamma(\frac{k+p}{2})}{\Gamma(\frac{k}{2})(k\pi)^{p/2}|\Sigma|^{1/2}}
\left(1 + \frac{(t-\mu)'\Sigma^{-1}(t-\mu)}{k}\right)^{-\frac{k + p}{2}}.
$$
For $k > 2$, the covariance matrix is given by $\frac{k}{k-2}\Sigma$.} when
the variance is estimated. We summarize this in the following corollary.

\begin{corollary} \label{cor:mincov_unknownvar}
For $\lambda_n = \gamma_n/\hat\sigma$ and if $M_n \subseteq \R^p$ is
non-random and satisfies Condition~\ref{cond:A}, we have that
$$
\inf_{\beta \in \R^p} P_\beta(\beta \in \betaL - \hat\sigma M_n) =
\min_{d \in \{-1,1\}^p} P(\hat u_n^d \in \hat\sigma M_n) =
\min_{d \in \{-1,1\}^p} P(\hat t_n^d \in M_n),
$$
where $\hat t_n^d \sim T(n-p,-n^{-1/2}C_n^{-1}\Gamma_n d,C_n^{-1})$ is a
multivariate non-central $t$ distribution with $n-p$ degrees of freedom,
non-centrality parameter $\mu = -n^{-1/2}C_n^{-1}L_n d$ where $\Gamma_n =
\diag(\gamma_n)$, and matrix $C_n^{-1}$.
\end{corollary}

One can now construct confidence regions in case where $\sigma^2$ is
unknown. Note that the shape of the contour sets of the above
$t$-distribution is the same as for the original distribution of $\hat
u_n$, namely $E_{C_n}(k) = \{z \in \R^p : z'C_nz \leq k\}$. Therefore, all
considerations from Section~\ref{sec:construct} also apply in this setting,
only the choice of the parameter $k$ needs to be adapted.

\subsection{Coverage probabilities over the parameter space}
\label{subsec:paramspace}

Since the derivation of Theorem~\ref{ther:infprob} intimates that the
minimal coverage probability occurs for ``large'' values of the unknown
parameter one might ask how the coverage looks for ``small'' values. As
explicit expressions for the coverage probability are not known, we give
plots of the simulated coverage probability of the 95\% Lasso ellipse for
$p=2$ for positive and negative correlation of the two components in
Figure~\ref{fig:coverage}.
\begin{figure}
\centering
\subfloat[]{\includegraphics[width=0.48\textwidth]{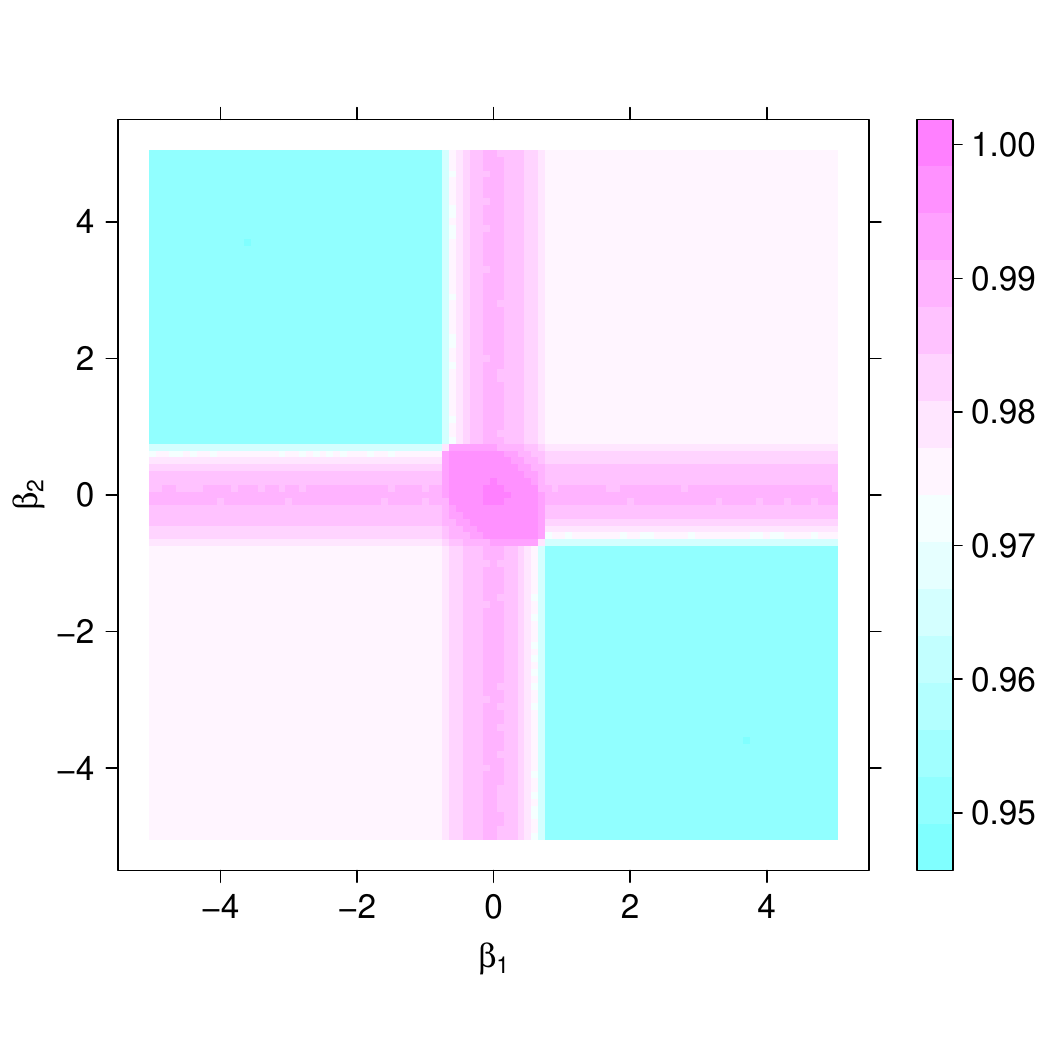}
\label{subfig:coverage1}}
$\;\;$
\subfloat[]{\includegraphics[width=0.48\textwidth]{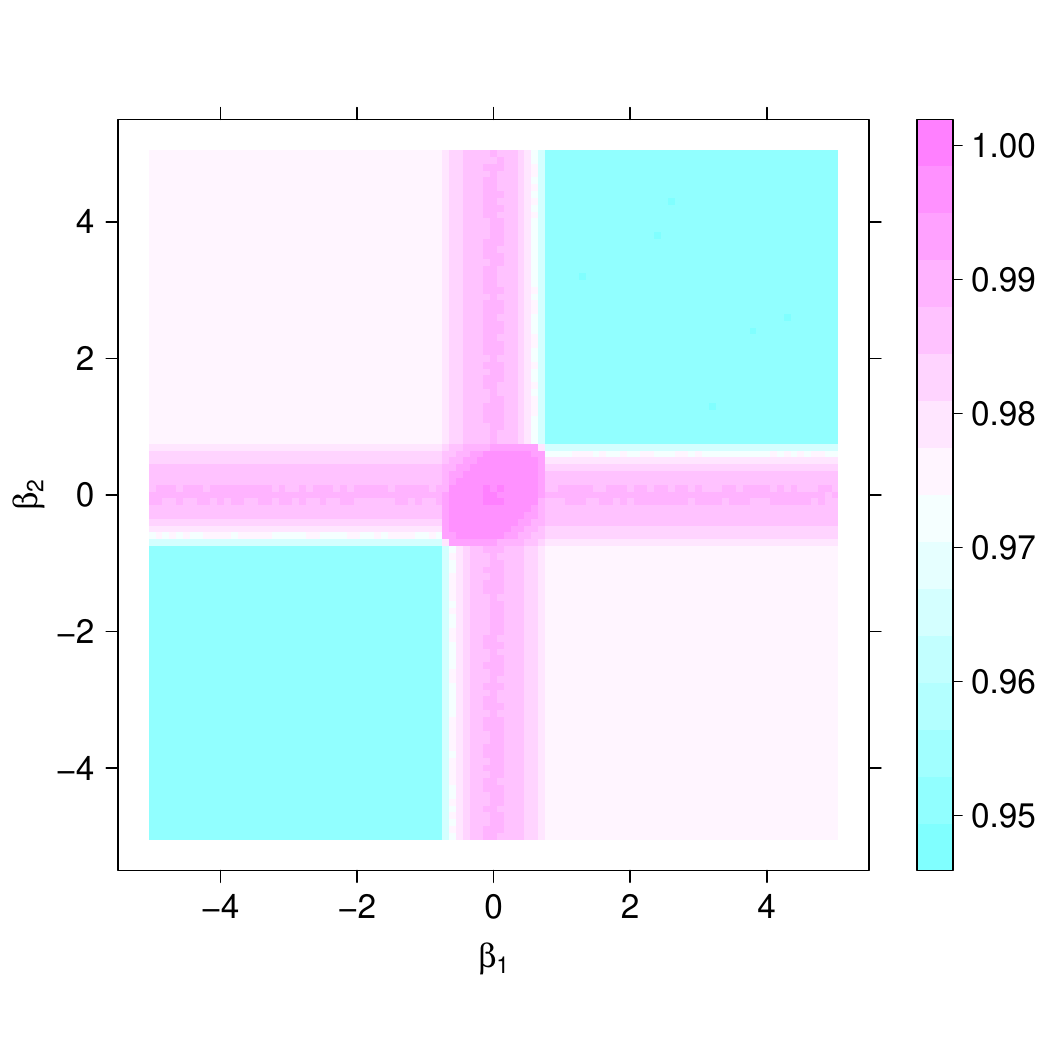}
\label{subfig:coverage2}}
\caption{\label{fig:coverage} The coverage probability of the Lasso
ellipse for $p=2$, $1\!-\!\alpha = 0.95$, $n=20$  and \protect\subref{subfig:coverage1} 
$C_n = 
\left(\protect\begin{smallmatrix} 1 & 0.5 \\ 0.5 & 1 \protect\end{smallmatrix}\right)$ 
and \protect\subref{subfig:coverage2} $C_n = 
\left(\protect\begin{smallmatrix} 1 & -0.5 \\ -0.5 & 1\protect\end{smallmatrix}\right)$.} 
\end{figure}
As can be seen, the minimal coverage occurs when the true parameter is
``not small''. More concretely, for the case of positive correlation it
occurs when both components are of opposite signs, and in case of negative
correlation it occurs when both components are of the same sign. It can
also be seen that in case the parameter space is known to be sparse (for
$p=2$, this means that at least one component of the parameter vector is
equal to 0), the minimal coverage over the restricted parameter space will
certainly be higher than the minimal coverage over the entire parameter
space. We cannot provide analytic expressions for minimal coverage
probability over a sparse parameter space using our theory and it covers
the case where no additional information about the parameter space is
available. It can, however, also be gleaned from Figure~\ref{fig:coverage}
that the common restriction of assuming that the true parameter is either
equal to zero or bounded away from zero (asymptotically at a certain rate)
does not alleviate the situation for the Lasso estimator!

\subsection{Inference on single components} \label{subsec:singlecomp}

We now consider the case where one might be interested in covering only a
subvector of the entire unknown parameter vector. While it is clear that
projecting the confidence region constructed for the entire parameter
vector to the appropriate subspace will yield a valid confidence set for
this purpose, it will generally not result in the most favorable shape.

In this subsection, we assume that the goal is to cover a single component
of the parameter vector and give general considerations on how to determine
the shape of the confidence region for the entire parameter vector so that
the projection onto the single component of interest will yield the
smallest symmetric\footnote{\cite{PoetscherSchneider10} show (for the case
of orthogonal regressors) that for single components, symmetric intervals
are the shortest, we therefore restrict ourselves also the symmetric case
here.} interval possible. 

More formally, consider the following. Assume that we want to construct a
confidence interval for $\beta_j$, the $j$-th component of the unknown
parameter vector $\beta$, with level of coverage $1 - \alpha$. For this, we
want to choose $M \subseteq \R^p$ such that

\begin{itemize}

\item $M$ satisfies Condition~\ref{cond:A}.

\item $\sup_{m \in M} |m_j| = a < \infty$.

\item $\inf_{d \in \{-1,1\}^p} P(\hat u_n^d \in M) = 1 - \alpha$.

\end{itemize}

Clearly, this can be achieved by finding for any fixed but arbitrary $a
\geq 0$ the largest set that satisfies Condition~\ref{cond:A} and then
choosing $a$ so that the prescribed coverage level is achieved. Note that
this set may be unbounded with respect to the components that are not of
interest. We construct the optimal shape for this explicitly for the case
where $p=2$, assuming that both components are penalized (both $\lambda_1$
and $\lambda_2$ are non-zero) in the following section.

\subsubsection*{Constructing the optimal shape in case $p=2$}

The following construction yields the set $M$ as described above for the
case of $p=2$. Without loss of generality, we assume that we are interested
in covering $\beta_1$, the first component of $\beta$. Recall that $C_n =
X'X/n$. If $C_n$ is diagonal, it is easily seen that the set
$$
\tilde M = \{z \in \R^2 : |z_1| \leq a\}
$$
complies with Condition~\ref{cond:A} and cannot be enlarged while
maintaining a fixed projection onto the subspace associated with the
first component. Also note that in this case, the resulting confidence
interval will coincide with the one suggested in \cite{PoetscherSchneider10}.

If $C_n$ is not diagonal, assume that the off-diagonal element $c_{12}$
satisfies $c_{12} > 0$\footnote{Otherwise construct a confidence interval
for $\beta_1$ from the model $y_i = \beta_1 x_{i1} + \tilde\beta_2 x_{i2} +
\eps_i$ where $\tilde\beta_2 = - \beta_2$ and $\tilde x_{i2} = -x_{i2}$.}.
Define
$$
M = \bigcup_{d \in \{-1,1\}^2} M^d
$$
with 
$$
M^{(1,1)} = \tilde M \cap \{z \in \R^2:  z_1,z_2 \geq 0, (C_n z)_1 \leq
(C_n\underbar a)_1\},
$$ 
where $\underbar a = (a,0)'$ and 
$$
M^{(-1,1)} = \tilde M \cap \{z \in \R^2: z_1 \leq 0, z_2 \geq 0, (C_n z)_2
\leq (C_n\underbar b)_2 \},
$$ 
where $\underbar b = (0,b)'$ satisfies $(C_n\underbar a)_1 = (C_n\underbar
b)_1$. Moreover, we define
$$
M^{(-1,-1)} = -M^{(1,1)} \;\; \text{ and } \;\; M^{(1,-1)} = -M^{(-1,1)}.
$$
The shape of the resulting set is depicted in Figure~\ref{fig:Minterval}.
\begin{figure}
\centering
\includegraphics[width=0.48\textwidth]{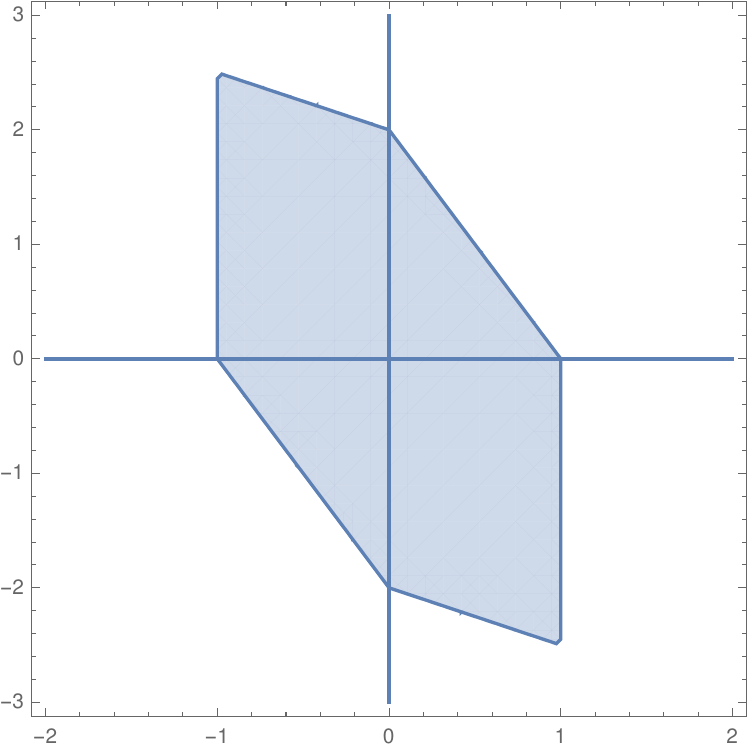}
\caption{\label{fig:Minterval} The set $\M$ for $C = \left(\protect
\begin{smallmatrix} 1 & 0.5 \\ 0.5 & 1 \protect\end{smallmatrix}\right)$ 
and $a = 1$.}
\end{figure}
Note that, even though we are only interested in a confidence set that is
bounded for one of the components, the need to comply with
Condition~\ref{cond:A} forces us to bound the set in the other component as
well whenever $c_{12} \neq 0$. The interpretation of this fact is the
following. As the Lasso can be viewed as a shifted LS estimator where the
size and direction of the shift depend on both components of the LS
estimator, we need to ensure that the influence of the second parameter on
the shift is also corrected for by the procedure.

The following proposition ensures that $M$ satisfies Condition~\ref{cond:A}
and does indeed yield the largest such set with fixed projection $[-a,a]$
onto the first component -- therefore providing the shape that results in
the smallest confidence interval for $\beta_1$.

\begin{proposition} \label{prop:Minterval}
The set $M \subseteq \R^2$ as defined above satisfies
Condition~\ref{cond:A}. Moreover, if another $\bar M \subseteq \R^2$
with $\max_{m \in \bar M} |m_1| \leq a$ satisfies
Condition~\ref{cond:A} also, $\bar M \subseteq M$ follows.
\end{proposition} 

It is again easily seen that for a given coverage probability $1-\alpha$,
the quantity $a$ must be greater than the half-length of the standard
interval based on the LS estimator, that is, the $(1-\alpha/2)$-quantile of
the standard normal distribution. One might now be interested in the size
difference between the confidence intervals constructed from the Lasso and
LS estimates, respectively. \cite{PoetscherSchneider10} have already shown
that in the orthogonal regressor case, the length of confidence intervals
which are based on the Lasso is greater than the length of the standard
interval and that the difference increases with the penalization parameter
$\lambda_n = (\lambda_{n,1},\lambda_{n,2})'$. Table~\ref{tab:halflengths}
contains the required values of $a$, that is, the half-lengths of the Lasso
confidence interval for $c_{11} = c_{22} = 1$, $\sigma^2 = 1$ and various
combinations of $\bar\lambda = \lambda_{n,1} = \lambda_{n,2}$ and $c_{12}$.
Note that in this case the LS estimator is the the Lasso estimator with
$\bar\lambda = 0$. \setlength{\extrarowheight}{2pt}
\begin{table}[htp]
\begin{center}
\begin{tabular}{|c|cccc|}
\hline
$|c_{12}|$ & 0.25 & 0.5 & 0.75 & 0.9 \\
\hline
$\bar\lambda = 0$ & 1.96 & 1.96 & 1.96 & 1.96 \\
$\bar\lambda = 0.1$ & 2.1 & 2.4 & 3.1 & 4.9 \\
$\bar\lambda = 0.5$ & 2.4 & 2.9 & 4.5 & 8.8 \\
$\bar\lambda = 1$ & 3.0 & 3.9 & 6.5 & 13.8  \\
$\bar\lambda = 2$ & 4.4 & 5.9 & 10.5 & 23.8 \\
$\bar\lambda = 3$ & 5.7 & 7.9 & 14.5 & 33.8 \\
\hline
\end{tabular}
\smallskip
\caption{\label{tab:halflengths} Half-lengths of the 95\% confidence
intervals based on an equally tuned Lasso estimator for and $c_{11} =
c_{22} = 1$ and $\sigma^2 = 1$.}
\end{center}
\end{table}
For small values of $\bar\lambda$ and $c_{12}$, the resulting confidence
interval is only slightly longer than the one based on the LS estimator.
For increasing $\bar\lambda$ and $|c_{12}|$, the required length of the
interval increases significantly, in particular in the latter case, with
the length more than doubling as $c_{12}$ increases from 0.25 to 0.9 for
each of the presented values of $\bar\lambda > 0$. This ratio is even more
extreme for larger values of $\bar\lambda > 0$. Two effects are at play
here. On the one hand, the area of $M$ decreases for fixed $a > 0$ as
$c_{12}$ increases. On the other hand, some of the corners of the distorted
$\lambda$-box, $-n^{-1/2}C_n^{-1}\Lambda_n d$ with $d \in \{−1,1\}^p$,
which are the means of the normal distributions that must be covered, shift
further apart as $c_{12}$ increases in absolute value. Obviously,
increasing the tuning parameter also shifts the means further away from the
origin, resulting in even larger confidence sets.

\section{Asymptotic framework} \label{sec:asymp}

We now derive asymptotic results that hold without assuming normality of
the errors. Additionally to the assumptions in Section~\ref{sec:setting},
\emph{for all asymptotic considerations}, we assume that $X =
(x_1',\dots,x_n')'$ where $x_i' \in \R^p$, meaning that the regressor
matrix $X$ changes with $n$ only by appending rows, and that
$$
C_n = \frac{X'X}{n} \longto C
$$
as $n \to \infty$, where $C$ is finite and positive definite. This setting
assures consistency and asymptotic normality of the LS estimator. We will
consider two different regimes of the asymptotic behavior of the tuning
parameter $\lambda_n$ and start with the regime we call \emph{conservative
tuning}.

\subsection{Conservative tuning} \label{subsec:conservative}

In this regime and \emph{throughout this subsection}, we require that
$$
\frac{\lambda_n}{n^{1/2}} \longto \lambda \in [0,\infty)^p
$$
as $n \to \infty$. This implies that $\lambda_{n,j}/n \to 0$ for all $j =
1,\dots,p$ , which in turn implies consistency of $\betaL$ (see Theorem~1
in \cite{KnightFu00} with the slight modification that in our paper we
allow for componentwise defined tuning parameters). We let $\Lambda =
\diag(\lambda)$.

\begin{remark} \label{rem:conservtuning}
Such a choice of tuning parameters indeed yields a conservative model
selection procedure in the sense that
\begin{equation}
\label{eq:conservmsp}
\limsup_{n \to \infty} 
\sup_{\beta \in \R^p} P_\beta \left(\hat\beta_j = 0\right) < 1
\end{equation}
for each $j = 1,\dots,p$. In particular, if $\beta_j = 0$, we have
$$
\limsup_{n \to \infty} P_\beta\left(\hat\beta_j = 0\right) < 1.
$$
The latter statement was also noted by \cite{Zou06} in Proposition 1.
\end{remark}
The following proposition implicitly states the asymptotic distribution of
the estimator in a so-called moving-parameter framework. This proposition
essentially is Theorem~5 from \cite{KnightFu00} and can be proven in the
same manner simply by adjusting for componentwise tuning.

\begin{proposition} \label{prop:Vconserv}
Assume that $n^{1/2}\beta_n \to t \in \Rquer^p$. Then $n^{1/2}(\betaL
- \beta_n) \dto \hat u = \argmin_{u \in \R^p} Q(u)$, where
\begin{equation} \label{eq:Vconserv}
Q(u) = u'Cu - 2 W'u + 2 \sum_{j=1}^p \lambda_j 
\left[\ind_{\{t_j \in \R\}} (|t_j + u_j| - |t_j|) + 
\ind_{\{|t_j| = \infty\}} \sgn(t_j) u_j\right]
\end{equation}
and $W \sim N(0, \sigma^2 C)$.
\end{proposition}

Note that the vector $t$ takes over the role of $n^{1/2}\beta$ in the
finite-sample version of the function, $Q_n$, where the cases of
$n^{1/2}\beta_j = \pm \infty$ are now included in the asymptotic setting.
Also, the assumption of $n^{1/2}\beta_n$ converging in $\Rquer^p$ is not a
restriction in the sense that, by compactness of $\Rquer^p$,
Proposition~\ref{prop:Vconserv} characterizes all accumulation points of
the distributions (with respect to weak convergence) corresponding to
completely arbitrary sequences of $\beta_n$.


Similarly to the finite-sample case, we define $\hat u$ to be the unique
minimizer of $Q$, and for $d \in \{-1,1\}^p$, we define $Q^d(u) = u'Cu - 2
W'u + 2\sum_{j=1}^p \lambda_j d_j u_j$ with unique minimizer $\hat u^d$. We
can then formulate an asymptotic version of Theorem~\ref{ther:infprob}.

\begin{theorem} \label{ther:infprobconserv}
If $M \subseteq \R^p$ satisfies Condition~\ref{cond:A} with $\bar C =
C$, then
$$
\inf_{t \in \Rquer^p} P_t \left(\hat u \in M \right) = \min_{d \in
\{-1, 1\}^p} P \left(\hat u^d \in M \right),
$$
where $\hat u^d \sim N(C^{-1}\Lambda d, \sigma^2 C^{-1})$.
\end{theorem}

Given this result we can, again, construct asymptotically valid confidence
sets for the parameter $\beta$ in the following way.

\begin{corollary} \label{cor:coverageconserv}
If $M \subseteq \R^p$ satisfies Condition~\ref{cond:A} with $\bar C =
C$ and $\min_{d \in \{-1,1\}^p} P\left(\hat u^d \in M \right) = 1 -
\alpha$, where $\hat u^d \sim N(C^{-1}\Lambda d,\sigma^2 C^{-1})$ then
$$
\liminf_{n \to \infty} \inf_{\beta \in \R^p} 
P \left(\beta \in \betaL - n^{-1/2}M\right) = 1 - \alpha.
$$
\end{corollary}

We find that asymptotically in the case of conservative tuning, we
essentially get the same results as in finite samples when assuming
normally distributed errors. The only difference is that the minimal
coverage holds asymptotically and that the quantities $C_n$ and
$n^{-1/2}\Lambda_n$ have settled to their limiting values $C$ and
$\Lambda$, respectively.

\subsection{Consistent tuning} \label{subsec:consist}

In the second regime and \emph{throughout this subsection}, we suppose
that
$$
\frac{\lambda_{n,j}}{n^{1/2}} \longto \infty
$$
\emph{for at least one} $j$ with $1 \leq j \leq p$ as well as
$$
\frac{1}{n} \lambda_{n, j} \longto 0
$$
\emph{for all} $j=1,\dots,p$ as $n \to \infty$, where the latter condition
ensures estimation consistency of the estimator. We refer to this regime as
\emph{consistent tuning} to highlight the contrast to conservative tuning
where $\lambda_{n,j}/n^{1/2}$ converges for each $j = 1,\dots,p$. Yet we
emphasize that in order to ensure $P_\beta(\hat\beta_{\L,j} = 0) \to 1$
whenever $\beta_j = 0$, we would need $\lambda_{n,j}/n^{1/2} \to \infty$
for each $j = 1,\dots,p$ as well as need additional conditions on the
regressor matrix $X$. We refer the reader to \cite{Zou06}, \cite{ZhaoYu06}
and \cite{YuanLin07b} for a discussion concerning necessary and sufficient
conditions on $X$ in this context.


In the case of consistent tuning, the rate of the estimator is no longer
$n^{-1/2}$, neither when looked at in a fixed-parameter asymptotic
framework (as has been noted by \cite{Zou06} in Lemma 3), nor (a fortiori)
within a moving-parameter asymptotic framework, as discussed in in
\cite{PoetscherLeeb09} in Theorem~2. The latter reference shows that the
correct (uniform) convergence rate depends on the sequence of tuning
parameters $\lambda_n$. Since we allow for componentwise tuning, in fact,
the rate depends on the largest component of the vector of tuning
parameters, as can be seen from the following proposition. We define
$$
\lambda^*_n = \max_{1 \leq j \leq p} \lambda_{n,j}
$$
and $\lambda_0 = (\lambda_{0,1},\dots,\lambda_{0,p})'$ by
$$
\lambda_{n,j}/\lambda_n^* \longto \lambda_{0,j}  \in [0,1]
$$
for each $j = 1,\dots,p$ as $n \to \infty$. Note that $\lambda_{0,j} =
1$ for all $j$ in case all components are equally tuned.

\begin{proposition} \label{prop:Vconsist} 
Assume that $n\beta_n/\lambda_n^* \to \zeta \in \Rquer^p$. Then
$n(\betaL - \beta)/\lambda^*_n \pto m = \argmin_{u \in \R^p}
V^\zeta(u)$, where
$$
V^\zeta(u) = u'Cu + 2\sum_{j=1}^p \lambda_{0,j} 
\left[\ind_{\{\zeta_j \in \R\}} (|u_j + \zeta_j| - |\zeta_j|) + 
\ind_{\{|\zeta_j| = \infty\}} \sgn(\zeta_j)u_j\right].
$$
\end{proposition}

(In contrast to the finite-sample and the conservative case, we make
the dependence of the objective function $V^\zeta$ on the unknown
parameter $\zeta \in \Rquer^p$ apparent in the notation to clarify
what we do in the following). Proposition~\ref{prop:Vconsist} shows
that $\lambda^*_n/n$ is indeed the correct (uniform) convergence rate
as the limit of $n(\betaL - \beta)/\lambda^*_n$ is not 0 in general.
The proposition also reveals that in the consistently tuned case, when
scaled according the correct convergence rate, the limit of the
sequence of estimators is always non-random, a fact that in a
moving-parameter asymptotic framework has already been noted in the
one-dimensional case in \cite{PoetscherLeeb09}. This fact allows us to
construct very simple confidence sets in the case of consistent tuning
by first observing that the limit of $n(\betaL - \beta)/\lambda^*_n$
is always contained in a bounded set which is described in
Proposition~\ref{prop:setscriptM}. To this end, define the set
\begin{equation}
\label{eq:M}
\M = \bigcup_{\zeta \in \Rquer^p} \argmin_{u \in \R^p} V^\zeta(u)
\end{equation}
and note that the following can be shown.

\begin{proposition} \label{prop:setscriptM} 
The set $\M$ can be written as
$$
\left\{m \in \R^p: |(Cm)_j| \leq \lambda_{0,j}, 1 \leq j \leq p
\right\} = C^{-1} \left\{z \in \R^p: |z_j | \leq \lambda_{0,j}, 1 \leq
p \right\}.
$$
\end{proposition}
Thus $\M$ can be viewed as a box distorted by the linear function
$C^{-1}$, a bounded set in $\R^p$. In fact, this turns out to be a
parallelogram whose corner points are given by the set $\{C^{-1}
\Lambda_0 d : d \in \{-1,1\}^p \}$, where $\Lambda_0 =
\diag(\lambda_0)$. Note that fittingly, these corner points can be
viewed as the equivalent of the means in the normal distributions
(determining the minimal coverage probability) in the conservative
case in Theorem~\ref{ther:infprobconserv}, appearing without
randomness in the limit in the consistently tuned case. Using
Proposition~\ref{prop:setscriptM}, a simple asymptotic confidence set
can now be constructed as is done in the following corollary.

\begin{corollary} \label{cor:coverageconsist}
We have
$$
\lim_{n \to \infty} \inf_{\beta \in \R^p} 
P_\beta\left(\beta \in \betaL - d\frac{\lambda_n^*}{n} \M\right) = 1
$$
for any $d > 1$ and
$$
\lim_{n \to \infty} \inf_{\beta \in \R^p} 
P_\beta\left(\beta \in \betaL - d\frac{\lambda_n^*}{n} \M\right) = 0
$$
for any $d < 1$.
\end{corollary}

Note that nothing can be said about the boundary case $d=1$. This
corollary is a generalization of the simple confidence interval given
in Proposition~6 in \cite{PoetscherSchneider10}. The shape of $\M$ is
depicted in Figure~\ref{fig:Mconsist}. Finally, also note the set $\M$
is not required to satisfy Condition~\ref{cond:A} and, in fact, will
not comply with this condition for certain matrices $C$.

\begin{figure}
\centering
\includegraphics[width=0.48\textwidth]{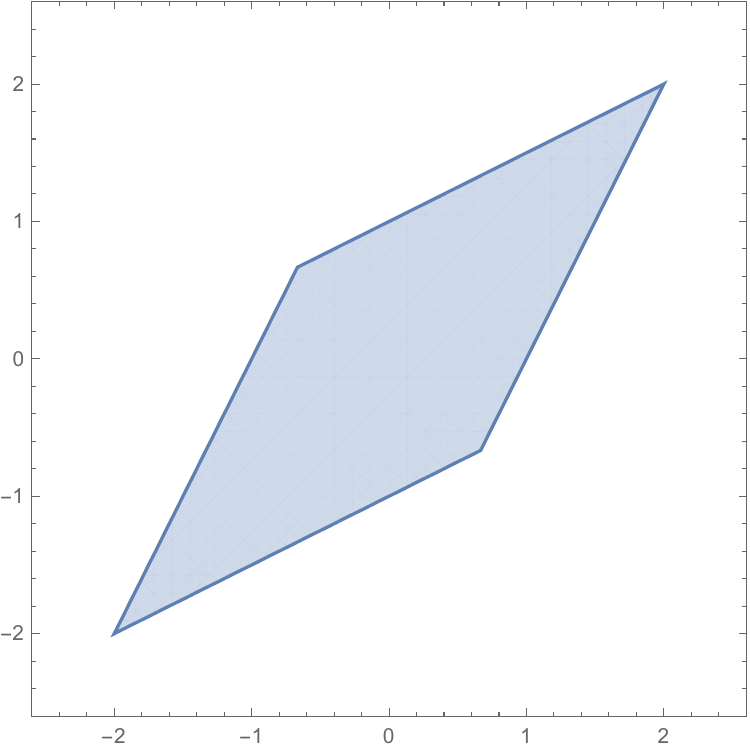}
\caption{\label{fig:Mconsist} The set $\M$ for $C = \left(\protect
\begin{smallmatrix} 1 & -0.5 \\ -0.5 & 1 \protect\end{smallmatrix}\right)$ 
and $\lambda_0 = (1,1)'$.}
\end{figure}

\section{Summary and conclusion} 
\label{sec:conclusion}

We consider confidence regions based on the Lasso estimator covering the
entire unknown parameter vector, thereby quantifying estimation uncertainty
of this estimator. We provide exact formulas for the minimal coverage
probability of these regions in finite samples and asymptotically in a
low-dimensional framework when the estimator is tuned to perform
conservative model selection. We do this without explicit knowledge of the
distribution but by carefully exploiting the structure of the optimization
problem that defines the estimator. The sets we consider as confidence
regions need to satisfy certain shape constraints which apply to the
regular confidence ellipse based on the LS estimator. We show that the LS
confidence ellipse is always smaller than the one based on the Lasso
estimator, but not contained in the Lasso ellipse in general. An ellipse is
not the optimal shape for the confidence region based on the Lasso
estimator in terms of volume. We give some guidelines on how to construct
regions of smaller volume. We show how a set can be minimally enlarged in
order to comply with the imposed shape condition, allowing to start the
construction with sets of arbitrary shapes.  We also illustrate how the
coverage probability of the Lasso ellipse varies over the parameter space
for the case when $p=2$, in which we also show how our results can be used
for constructing valid confidence intervals for single components of the
parameter space. In case the error variance needs to be estimated, our
results involve non-central $t$-distributions rather than shifted normal
distributions. Finally, in the consistently tuned case, we give a simple
asymptotic confidence regions in the shape of a parallelogram that is
determined by the regressor matrix.

\begin{appendix}
\section{Proofs} \label{sec:proofs}

We start the proof section with introducing some notation that will be used
throughout this section. Let $e_j$ denote the $j^\th$ unit vector in $\R^p$
and let $\iota = (1,\dots,1)' \in \R^p$. For a vector $d \in \{-1,1\}^p$,
we define $\O^d$ to be the corresponding orthant of $\R^p$, that is, $\O^d
= \{z \in \R^p :  d_jz_j \geq 0\}$ and $\bar\O^d$ to be the corresponding
orthant of $\Rquer^p$, that is, $\bar\O^d = \{z \in \Rquer^p : d_jz_j \geq
0\}$. By $\O^\iota_\Int$ we denote the orthant with strictly positive
components only, that is, $\O^\iota_{\Int} = \{z \in \R^p : z_j > 0\}$. The
sup-norm on $\R^p$ is denoted by $\|.\|_\infty$.


To remind the reader of some notation relevant for the following proofs
that was introduced previously throughout the paper, note that $\hat u_n =
n^{1/2}(\betaL - \beta)$, where $\hat u_n$ is the minimizer of $Q_n$, and
$\uLS = n^{1/2}(\betaLS - \beta)$. The minimizer of $Q_n^d$ was labeled
$\hat u_n^d$. The asymptotic versions in the conservatively tuned case were
labeled $\hat u$ and $Q$, as well as $\hat u^d$ and $Q^d$, respectively.


The directional derivative of a function $g:\R^p \to \R$ at $u$ in the
direction of $r \in \R^p\setminus{\{0\}}$ is defined as
$$
\pd{g(u)}{r} = \lim_{h \searrow 0} \frac{g(u + hr) - g(u)}{h}.
$$

\subsection{Proofs for Section~\ref{sec:finitesample}} \label{sec:proofsfs}

In order to prove the main theorem, we start by re-writing
Condition~\ref{cond:A}. For $m \in \R^p$ and a $p \times p$ matrix
$\bar C$, we define
\begin{align*}
A^{d_j}_{\bar C,j}(m) & \; = \; \{z \in \R^p: d_j(\bar Cm)_j \leq d_j(\bar Cz)_j,
d_jz_j \leq 0\} \text{ and } \\ 
B^{d_j}_{\bar C,j}(m) & \; = \; \{z \in \R^p: (\bar Cz)_j = (\bar Cm)_j, d_jz_j > 0\}
\end{align*}
for $j = 1,\dots,p$. Note that clearly we have
$$
A^d_{\bar C}(m) = \bigcap_{j=1}^p A^{d_j}_{\bar C,j}(m),
$$
and that, in fact, also the following lemma holds.


\begin{lemma} \label{lem:condA}
$$
\bigcup_{d \in \{-1,1\}^p} \; \bigcap_{j=1}^p A^{d_j}_{\bar C,j}(m) = 
\bigcup_{d \in \{-1,1\}^p} \;
\bigcap_{j=1}^p A^{d_j}_{\bar C,j}(m) \cup B^{d_j}_{\bar C,j}(m)
$$
\end{lemma}


\begin{proof}
We fix $m$ and $\bar C$, drop the corresponding subscripts and show
that the set on the left-hand side of the equation contains the set on
the right-hand side of the equation. To this end, take any $z$ from
the set on right-hand side. Then there exists a $d \in \{-1,1\}^d$
such that for each $j=1,\dots,p$, $z$ is either contained in $A^{d_j}_j$
or in $B^{d_j}_j$. We pick $f \in \{-1,1\}^p$ in the following way: if $z
\in A^{d_j}_j$, set $f_j = d_j$ and if $z \in B^{d_j}_j$, set $f_j = -d_j$.
Then, by construction, $z \in A^f_j$ for all $j=1,\dots,p$ and
therefore $z \in \bigcap_j A^f_j$ so that $z$ is contained in the set
on the left-hand side of the equation.
\end{proof}


Since needed later on, we also prove the following proposition which
quantifies the maximal distance between the Lasso and the LS estimator
in finite samples.

\begin{proposition} \label{prop:lsdifflasso} 
For each $j = 1,\dots,p$, we have
$$
\left|(X'X(\betaL - \betaLS))_j\right| \leq \lambda_{n,j},
$$
or, equivalently,
$$
\label{eq:lsdifflassou}
\left|\left(C_n(\hat u_n - \uLS)\right)_j\right| \leq n^{-1/2}\lambda_{n,j},
$$
where $\uLS = n^{1/2}(\betaLS - \beta)$.
\end{proposition}

\begin{proof}
The two inequalities above just differ by a scaling factor. We show
the latter one. We have $W_n = n^{-1/2}X'\eps = C_n\uLS$. Consider the
directional derivative  of $Q_n$ at its minimizer $\hat u_n$ in the
direction of $e_j$ and $-e_j$. We have
\begin{align*}
0 \; \leq \; \pd{}{e_j}Q_n(\hat u_n) & \; = \; 2(C_n\hat u_n)_j - 2W_{n,j} + 
2 n^{-1/2} \lambda_{n,j} \left[\ind_{\{\hat u_j \geq - n^{1/2}\beta_j\}} 
 - \ind_{\{\hat u_j < - n^{1/2}\beta_j\}}\right] \\
& \; \leq \; 2(C_n\hat u_n)_j - 2(C_n\uLS)_j + 2n^{-1/2}\lambda_{n,j}, 
\end{align*} 
as well as
\begin{align*}
0 \; \leq \; \pd{}{(-e_j)}Q_n(\hat u_n) & \; = \; -2(C_n\hat u_n)_j + 2W_{n,j} + 
2n^{-1/2}\lambda_{n,j} \left[\ind_{\{\hat u_j \leq - n^{1/2}\beta_j\}} 
 - \ind_{\{\hat u_j > - n^{1/2}\beta_j\}}\right] \\
& \; \leq \; -2(C_n\hat u_n)_j + 2(C_n\uLS)_j + 2n^{-1/2}\lambda_{n,j}.
\end{align*} 
Piecing the two displays above together yields the second inequality
in the proposition.
\end{proof}


To proceed note that $Q_n^d$ as defined in (\ref{eq:Qnd}) is a simple
quadratic and strictly convex function in $u$ with unique minimizer
$\hat u_n^d$ given by
\begin{equation} \label{eq:und}
\hat u_n^d = C_n^{-1} (W_n  - n^{-1/2}\Lambda_n d),
\end{equation}
where $W_n \sim N(0,\sigma^2 C_n)$. We first show
Theorem~\ref{ther:infprob} for one orthant of the parameter space
$\R^p$, as is formulated in Proposition~\ref{prop:infprobiota}.


\begin{proposition} \label{prop:infprobiota}
If $M_n \subseteq \R^p$ satisfies that 
$$
\bigcap_{j=1}^p A^{\iota_j}_{C_n,j}(m) \cup B^{\iota_j}_{C_n,j}(m) \subseteq M_n
$$
for all $m \in M_n$, then
$$
\inf_{\beta \in \O^\iota}
P_\beta(\hat u_n \in M_n) = P(\hat u_n^{\iota} \in M_n).
$$
\end{proposition}


In essence, Proposition~\ref{prop:infprobiota} states
Theorem~\ref{ther:infprob} for the orthant of the parameter space
where all components of $\beta$ are non-negative. The condition in
Proposition~\ref{prop:infprobiota} takes the role of
Condition~\ref{cond:A} for the corresponding orthant, as will become
apparent later on in the proof of Theorem~\ref{ther:infprob}.


\begin{proof}[Proof of Proposition~\ref{prop:infprobiota}]
We first show that $\inf_{\beta \in \O^\iota} P_\beta (\hat u_n \in
M_n) \geq P(\hat u_n^\iota \in M_n)$ by showing that for each fixed
$\omega\in \Omega$, $\hat u_n^\iota \in M_n$ implies that $\hat u_n
\in M_n$ as long as $\beta_j \geq 0$ for all $j$. For this, we first
show the following two facts.

\begin{enumerate}[ref=\emph{\alph*}]

\item \label{item:facta} $(C_n\hat u_n^\iota)_j \leq (C_n\hat u_n)_j$
for all $j = 1,\dots,p$.


Suppose there exists a $j_0$ with such that $(C_n\hat u_n^\iota)_{j_0}
> (C_n\hat u_n)_{j_0}$ and note that by (\ref{eq:und}) we have
$(C_n\hat u_n^\iota)_j = W_{n,j} - n^{-1/2}\lambda_{n,j}$ for each $j
= 1,\dots,p$. Now consider the directional derivative of $Q_n$ at its
minimizer $\hat u_n$ in direction $e_{j_0}$,
\begin{align*}
\pd{Q_n(\hat u_n)}{e_{j_0}} & \; = \; 2(C_n \hat u_n)_{j_0} -
2W_{n,j_0} + 2n^{-1/2} \lambda_{n,j_0} \left[
\ind_{\{\hat u_{n,j_0} \geq - n^{1/2}\beta_{j_0} \}} 
 - \ind_{\{\hat u_{n,j_0} < - n^{1/2}\beta_{j_0} \}}\right] \\
& \; \leq \; 2(C_n\hat u_n)_{j_0} - 2W_{n,j_0} + 2n^{-1/2}\lambda_{n,j_0} \\[1.5ex] 
& \; = \; 2(C_n\hat u_n)_{j_0} - 2 (C_n\hat u_n^\iota)_{j_0} < 0,
\end{align*}
which is a contradiction to $\hat u_n$ minimizing $Q_n$.

\item \label{item:factb} $\hat u_{n,j} > 0$ implies $(C_n\hat u_n)_j =
(C_n\hat u_n^\iota)_j$ for any $1 \leq j \leq p$.


If $\hat u_{n,j} > 0$ (and hence $\hat u_{n,j} + n^{1/2}\beta_j > 0$
when $\beta_j \geq 0$), then $Q_n$ is partially differentiable at
$\hat u_n$ with respect to the $j^\th$ component. Therefore, we have
\begin{align*}
\pd{Q_n(\hat u_n)}{u_j}	& \; = \; 
2(C_n\hat u_n)_j - 2W_{n,j} + 2n^{1/2}\lambda_{n,j} \\ 
& \; = \; 2(C_n\hat u_n)_j - 2(C_n\hat u_n^\iota)_j = 0.
\end{align*}

\end{enumerate}

Now, by Facts ({\ref{item:facta}}) and (\ref{item:factb}) we clearly
have that $\hat u_n \in A_{C_n}^{\iota_j}(\hat u_n^\iota) \cup
B_{C_n}^{\iota_j}(\hat u_n^\iota)$. So, by assumption, $\hat u_n^\iota
\in M_n$ clearly implies $\hat u_n \in M_n$ as long as $\beta_j \geq
0$ for all $j$. We have therefore shown that
$$
\inf_{\beta \in \O^\iota} P_\beta(\hat u_n \in M_n) \geq P(\hat u^\iota_n \in M_n).
$$

To see the reverse inequality, note that \emph{if} $\hat u_{n,j} +
n^{1/2}\beta_j > 0$ for all $j$, then $Q_n$ is differentiable at $\hat
u_n$ and
$$
\pd{Q_n(\hat u_n)}{u} = 2C_n\hat u_n - 2W_n + 2n^{-1/2}\lambda_n =
2C_n\hat u_n - 2C_n \hat u_n^\iota = 0,
$$
implying that $\hat u_n = \hat u_n^\iota$. Also note that $\hat
u_{n,j} + n^{1/2}\beta_j > 0$ for each $j$ is equivalent to $\betaL
\in \O^\iota_\Int$, so that
$$
\{\hat u_n \in M_n\} \subseteq \{\hat u_n^\iota \in M_n\} \cup
\{\betaL \notin \O^\iota_\Int\}.
$$
Now let $\kappa$ be a bound in the sup-norm on the set $\{z \in \R^p:
\|C_nz\|_\infty \leq n^{-1/2}\|\lambda_n\|_\infty\}$ and for an
arbitrary $\eps > 0$, pick $\beta^* \in \R^p$ such that $P(\uLS \leq
\kappa\iota - n^{1/2}\beta^*) \leq \eps$, where $\uLS =
n^{1/2}(\betaLS - \beta^*) \sim N(0,\sigma^2C_n^{-1})$. Note that by
Proposition~\ref{prop:lsdifflasso}, this implies that
$$
P_{\beta^*}(\betaL \leq 0) = 
P_{\beta^*}(\hat u_n - \uLS + \uLS \leq -n^{1/2}\beta^*) \leq
P_{\beta^*}(-\kappa\iota + \uLS \leq -n^{1/2}\beta^*) \leq \eps,
$$
yielding
$$
\inf_{\beta \in O^\iota} P_\beta(\hat u_n \in M_n) \leq
P_{\beta^*}(\hat u_n\in M_n) \leq P(\hat u^\iota_n \in M_n) + \eps.
$$
Since $\eps > 0$ was arbitrary, this shows the desired inequality.
\end{proof}


Essentially, we have now shown the main theorem for one part of the
parameter space $\R^p$. By flipping signs, we can apply
Proposition~\ref{prop:infprobiota} to each orthant $\O^d$, thus
obtaining the formula for the infimal coverage over the whole space.


\begin{proof}[Proof of Theorem~\ref{ther:infprob}]
First note that 
$$
\inf_{\beta \in \R^p} P_\beta (\hat u_n \in M_n) = \min_{d \in \{-1,1\}^p}
\inf_{\beta \in \O^d} P_\beta(\hat u_n \in M_n).
$$
Thus, if we can show that 
$$
\inf_{\beta \in \O^d} P_\beta(\hat u_n \in M_n) = P(\hat u_n^d \in
M_n)
$$
for each $d \in\{-1,1\}^p$, the proof is done. Now, fix $d$ and set $D
= \diag(d)$. We consider the function
\begin{align*}
\tilde Q_n(u) = Q_n(D u) & \; = \; u' DC_nDu - 2u'DW_n + 
2n^{-1/2}\sum_{j=1}^p\lambda_{n,j}\left[|d_j u_j + n^{1/2}\beta_j| -
|n^{1/2}\beta_j|\right] \\
& \; = \; u' \tilde C_n u - 2 u'\tilde W_n + 
2n^{-1/2}\sum_{j=1}^p\lambda_{n,j}
\left[|u_j + n^{1/2} d_j \beta_j| - |n^{1/2} d_j \beta_j|\right],
\end{align*}
where $\tilde C_n = D C_n D$, $\tilde W_n = DW_n \sim N(0,\sigma^2
\tilde C_n)$. We write $\tilde u_n$ for the minimizer of $\tilde Q_n$,
and, analogously to Section~\ref{sec:finitesample}, we define $\tilde
u_n^\iota$ to be the minimizer of the function $u'\tilde C_n u - 2
u'\tilde W_n + 2n^{-1/2}\sum_{j=1}^p\lambda_{n,j} u_j$.

If we can show that the set $DM_n$ satisfies the requirement of
Proposition~\ref{prop:infprobiota} with the matrix $\tilde C_n$ in
place of $C_n$, we may conclude that
$$
\inf_{\beta: d_j\beta_j \geq 0} P_\beta(\tilde u_n \in DM_n) = P(\tilde u_n^\iota
\in DM_n).
$$
Note that $\hat u_n = D\tilde u_n$, $\hat u_n^d = D\tilde
u_n^\iota$ and $D^{-1} = D$, so that
$$
\inf_{\beta \in \O^d} P(\hat u_n \in M_n) = 
\inf_{\beta \in \O^d} P(\tilde u_n \in DM_n) = 
P(\tilde u_n^\iota \in DM_n) = P(\hat u_n^d \in M_n),
$$
which proves the formula for the infimal coverage probability. We now
show that the set $D M_n$ satisfies that
$$
\bigcap_{j=1}^p A^\iota_{\tilde C_n,j}(Dm) \cup B^\iota_{\tilde C_n,j}(Dm) 
\subseteq DM_n
$$
for all $m \in M_n$. A straightforward calculation shows that this is
equivalent to
$$
\bigcap_{j=1}^p A^{d_j}_{C_n,j}(m) \cup B^{d_j}_{C_n,j}(m) \subseteq M_n
$$
for each $m \in M$ which clearly holds by Condition~\ref{cond:A} and
Proposition~\ref{lem:condA}.


The distributional result on $\hat u_n^d$ immediately follows by
(\ref{eq:und}).
\end{proof}

\subsection{Proofs for Section~\ref{sec:construct}} 

\begin{proof}[Proof of Proposition~\ref{prop:Cellipse}]
Let $m \in E_{C_n}(k)$ and $y \in A_{C_n}^d(m)$. We show that $y \in
E_{C_n}(k)$. Remember that $D = \diag(d)$ satisfies $DD=I_p$. Since $y
\in A_{C_n}^d(m)$ we have $-Dy \in \O^\iota$ and $-DC(m-y) \in \O^\iota$
implying that 
$$
y'C(m-y) = (Dy)'DC(m-y) \geq 0.
$$ 
Furthermore, since $(m-y)'C(m-y) \geq 0$, we have
$$
m'C(m-y) \geq y'C(m-y) \geq 0,
$$
which in turn yields
$$
m'Cm \geq m'Cy \geq y'Cy \geq 0.
$$
But this means that $k \geq m'Cm \geq m'Cy \geq y'Cy$ and therefore $y
\in E_{C_n}(k)$.
\end{proof}


\begin{proof}[Proof of Proposition~\ref{prop:dstar}]
We transform the ellipse to a sphere and the corresponding normal
distribution to have independent components with equal variances.
$$
P\left(\hat u^d_n \in E_{C_n}(k)\right) = 
P\left(C_n^{1/2}\hat u^d_n \in C_n^{1/2}E_{C_n}(k)\right),
$$
where $C_n^{1/2}\hat u^d_n \sim N(-n^{-1/2}C_n^{-1/2}\Lambda_n d,\sigma^2
I_p)$ and $C_n^{1/2}E_{C_n}(k) = \{z \in \R^p: \|z\|^2 \leq k\}$. So
clearly, the smallest probability will be achieved for the
distribution with mean furthest away from the origin, which is any
$d^*$ maximizing $\|C_n^{-1/2} \Lambda_n d\|$ over all $d \in
\{-1,1\}^p$.
\end{proof}


\begin{proof}[Proof of Proposition~\ref{prop:kstar}]
Similarly to the proof of Proposition~\ref{prop:dstar}, note that
$$
P\left(\hat u^d_n \in E_{C_n}(k)\right) = 
P\left(C_n^{1/2}\hat u^d_n/\sigma \in C_n^{1/2}E_{C_n}(k)/\sigma\right)
$$
with $\hat w = C_n^{1/2}\hat u^d_n/\sigma \sim
N(-n^{-1/2}C_n^{-1/2}\Lambda_n d/\sigma,I_p)$ and
$C_n^{1/2}E_{C_n}(k)/\sigma = \{z \in \R^p: \|z\|^2 \leq k/\sigma^2\}$.
Therefore, the probability in the above display is given by
$$
P(\|\hat w\|^2 \leq k/\sigma^2)
$$
where $\|\hat w\|^2$ clearly follows the claimed non-central
$\chi^2$-distribution.
\end{proof}


\begin{proof}[Proof of Proposition~\ref{prop:closeA}]
We start by showing that for any $m \in \R^p$, $d \in \{-1,1\}^p$, we have
\begin{equation}
\label{eq:closeA}
A^d_{\bar C}(y) \subseteq A^d_{\bar C}(m) \;\;\; \text{for all } 
y \in A^d_{\bar C}(m).
\end{equation}
Let $z \in A^d_{\bar C}(y)$. Then $d_j z_j \leq 0$ and $(\bar Cy)_j
\leq (\bar Cz)_j$ for all $j$. But since $y \in A^d_{\bar C}(m)$, we
also have $(\bar Cm)_j \leq (\bar Cy)_j$ for all $j$ so that that
$(\bar Cm)_j \leq (\bar Cz)_j$ for all $j$ and therefore $z \in
A^d_{\bar C}(m)$, thus proving (\ref{eq:closeA}). So clearly, the set
$$
\bigcup_{m \in M} \bigcup_{d \in \{-1,1\}^p} A^d_{\bar C}(m)
$$
satisfies Condition~\ref{cond:A}. For each $m \in M$, choose $d \in
\{-1,1\}^p$ in such a way that $d_j = 1$ if $m_j = 0$ and $d_j =
-\sgn(m_j)$ for $m_j \neq 0$. We then get $m \in A^d_{\bar C}(m)$,
implying that the set in the display above actually contains $M$.
\end{proof}
\subsection{Proofs for Section~\ref{sec:extensions}} \label{sec:proofsext}

\begin{proof}[Proof of Proposition~\ref{prop:Minterval}]
We start by proving that $M$ satisfies Condition~\ref{cond:A}. For this, we
need to show that for $d \in \{-1,1\}^2$, we have $A_{C_n}^d(m) =
A_{C_n}^d(m) \cap \O^{-d} \subseteq M^{-d}$ for all $m \in M$. We start by
doing so $d = (1,1)'$. Note that
$$
A^{(1,1)}_{C_n}(m) = \{z \in \O^{-(1,1)}: (Cm)_j \leq (Cz)_j, j=1,2\}
$$
and
$$
M^{-(1,1)} = \{m \in \O^{-(1,1)} : -(C\underbar{a})_1 \leq (Cm)_1\}
\subseteq \tilde M.
$$
If $m \in M^{-(1,1)}$, then clearly $-(C\underbar{a})_1 \leq (Cm)_1 \leq
(Cz)_1$ for any $z \in A^{(1,1)}_{C_n}(m)$, so that $z \in M^{-(1,1)}$
follows. If $m \in M^{(1,1)}$, then $A^{(1,1)}_{C_n}(m) = \emptyset$ unless
$m=0$, in which case $A^{(1,1)}_{C_n}(0) = \{0\}$. In either case,
$A^{(1,1)}_{C_n}(m) \subseteq M^{-(1,1)}$ follows immediately. If $m \in
M^{(-1,1)} \subseteq \{m \in \R^2: -a \leq m_1 \leq 0$, $m_2 \leq 0\}$, we
have $-(C\underbar a)_1 = -c_{11}a \leq c_{11}m_1 + c_{12}m_2 = (Cm)_1 \leq
(Cz)_1$ for any $z \in A^{(1,1)}_{C_n}(m)$, so that $z \in M^{-(1,1)}$
again follows. Finally, if $m \in M^{(1,-1)} \subseteq \{m \in \O^{(1,-1)}
: -c_{11}c_{22}a/c_{12} \leq (Cm)_2\}$, we have $-c_{11}a \leq
c^2_{12}/(c_{11}c_{22})\, c_{11}m_1 + c_{12}m_2 \leq (Cm)_1 \leq (Cz)_1$
for any $z \in A^{(1,1)}_{C_n}(m)$, so that $z \in M^{-(1,1)}$ follows yet
again.

The remaining cases $d = -(1,1)'$, $d = (-1,1)'$ and $d = (1,-1)'$ can be
shown in a similar manner.

To show the second part of Proposition~\ref{prop:Minterval}, assume there
exists $\bar m \in \bar M$ with $\bar m \notin M$ and show that this
implies $\max_{m \in \bar M}|m_1| > a$ if $\bar M$ complies with
Condition~\ref{cond:A}. If $\bar m \in \O^{(1,1)}$, then $\bar m \notin
M^{(1,1)}$ entails that $c_{11}\bar m_1 + c_{12}\bar m_2 = (C\bar m)_1 >
(C\underbar a)_1 = c_{11} a$. Let $\bar{\underbar{a}} = (\bar a,0)'$ where
$\bar a = \bar m_1 + c_{12}\bar m_2/c_{11} > a$ and note that
$\bar{\underbar{a}} \in A^{-(1,1)}_{C_n} \subseteq \bar M$ which implies
that $\max_{m \in \bar M}|m_1| \geq \bar a > a$ The remaining cases $\bar m
\in \O^{-(1,1)}$, $\bar m \in \O^{(-1,1)}$ and $\bar m \in \O^{(1,-1)}$ can
be shown in a similar manner.
\end{proof}

\subsection{Proofs for Section~\ref{sec:asymp}} \label{sec:proofsasymp}

\begin{proof}[Proof of Remark~\ref{rem:conservtuning}]

We show (\ref{eq:conservmsp}). Note that
Proposition~\ref{prop:lsdifflasso} entails that
$$
\betaL \in \betaLS - \frac{1}{n^{1/2}}B_n,
$$
where 
$$
B_n = \{z \in \R^p : |(C_nz)_j| \leq n^{-1/2}\lambda_{n,j} \text{ for }
j=1,\dots,p\}.
$$ 
Since $\lambda_n/n^{1/2}$ converges, we have $B_n \subseteq
C_n^{-1}\bar B_\delta$ with $\bar B_\delta = \{x \in \R^p :
\|x\|_\infty \leq \delta\}$ for some $\delta > 0$. Since $C_n^{-1} \to
C^{-1}$, the set $\{C_n^{-1} : n \in \N\}$ is bounded in operator
sup-norm by Banach-Steinhaus, so that the set $B_n$ is uniformly
bounded over $n$ in sup-norm by, say, $\gamma > 0$. We now fix a
component $j$ and show that $\liminf_{n \to \infty} \inf_{\beta \in
\R^p}P_\beta(\hat\beta_{\L,j} \neq 0) > 0$. To this end, define
$\mathcal{R}_j = \R^{j-1} \times \{0\} \times \R^{p-j}$. Let
$\xi^2_{j,n}$ and $\xi^2_j$ be the positive $j^\th$ diagonal element
of $C_n^{-1}$ and $C^{-1}$, respectively. Observe that
\begin{align*}
\inf_{\beta \in \R^p} P_\beta(\hat\beta_{\L,j} \neq 0) & \; \geq \;
\inf_{\beta \in \R^p} 
P_\beta\left((\betaLS - \frac{1}{n^{1/2}}B_n) \cap \mathcal{R}_j = \emptyset\right) \\ 
& \; \geq \; \inf_{\beta \in \R^p} P_\beta(n^{1/2}\hat\beta_{\LS,j} + \gamma < 0 
\text{ or } n^{1/2}\hat\beta_{\LS,j} - \gamma > 0) \\
& \; = 2 \Phi(-\gamma/\xi_{i,n}) \longto 2 \Phi(-\gamma/\xi_i) > 0
\end{align*}
\end{proof}


In order to prove Theorem~\ref{ther:infprobconserv}, we need an
asymptotic version of Proposition~\ref{prop:infprobiota} which is
formulated in the following.


\begin{proposition}\label{prop:infprobiotaasymp}
If $M \subseteq \R^p$ satisfies that 
$$
\bigcap_{j=1}^p A^{\iota_j}_{C,j}(m) \cup B^{\iota_j}_{C,j}(m) \subseteq M
$$
for all $m \in M_n$, then
$$
\inf_{t \in \bar\O^\iota}
P_t(\hat u \in M) = P(\hat u^{\iota} \in M).
$$
\end{proposition}


\begin{proof} 
The first part of the proof is completely analogous to the first part
of the proof of Proposition~\ref{prop:infprobiota} after identifying
$n^{1/2}\beta$ with $t$ and dropping the subscript $n$. 
To see the reverse inequality, note that for $t^* =
(\infty,\dots,\infty) \in \Rquer^p$, we actually have $Q=Q^\iota$, so
that $\hat u = \hat u^\iota$ in this case which already yields that
$$
\inf_{t \in \bar O^\iota} P_t(\hat u \in M) \leq P_{t^*}(\hat u \in M) = 
P(\hat u^\iota \in M).
$$
\end{proof}


\begin{proof}[Proof of Theorem~\ref{ther:infprobconserv}]
The proof again is completely analogous to the proof of
Theorem~\ref{ther:infprob} after identifying $n^{1/2}\beta$ with $t$,
dropping the subscript $n$ everywhere and using
Proposition~\ref{prop:infprobiotaasymp} instead of
Proposition~\ref{prop:infprobiota}. Also, replace $\O^d$ by $\bar\O^d$
and note that 
\begin{align*}
Q^d(u) & \; = \; Q(Du) = u'DCDu - 2u'DW + 2\sum_{i=1}^p
\lambda_j \left[\ind_{\{t_j \in \R\}} (|t_j + d_ju_j| - |t_j|)
+ \ind_{\{|t_j| = \infty\}} \sgn(t_j) d_ju_j\right] \\
& \; = \; u'\tilde Cu - 2u'\tilde W + 2\sum_{i=1}^p \lambda_j 
\left[\ind_{\{d_jt_j \in \R\}} (|u_j + d_jt_j| - |d_jt_j|) + \ind_{\{|d_jt_j| =
\infty\}} \sgn(d_jt_j) u_j\right],
\end{align*}
where $\tilde C = DCD$ and $\tilde W = DW$.
\end{proof}


\begin{proof}[Proof of Corollary~\ref{cor:coverageconserv}] 
Let $c =  \liminf_{n \to \infty} \inf_{\beta \in \R^p}P_\beta(\beta
\in \betaL - n^{-1/2}M)$. Then there exists a sequence $\beta_n$ in
$\R^p$ such that $P_{\beta_n}(\beta_n \in \betaL - n^{-1/2}M) \to c$.
Assume that $n^{1/2}\beta_n \to t \in \Rquer^p$ (if the sequence does
not converge, pass to subsequences). Since
$$
P_{\beta_n}(\beta_n \in \betaL - n^{-1/2}M) =
P_{\beta_n}(n^{1/2}(\betaL - \beta_n) \in M) \longto c = P_t(\hat u
\in M)
$$
as $n \to \infty$ in the notation of Proposition~\ref{prop:Vconserv}.
Theorem~\ref{ther:infprobconserv} then yields $c \geq \min_{d \in
\{-1,1\}^p} P(\hat u^d \in M) = 1 - \alpha$. To see the reverse
inequality, let $\beta_n = d \in \{-1,1\}^p$ and note that for this
sequence, we have
$$
P_{\beta_n}(\beta_n \in \betaL - n^{-1/2}M) =
P_{\beta_n}(n^{1/2}(\betaL - \beta_n) \in M) \longto P_t(\hat u \in M)
$$
as $n \to \infty$, where $t = (d_1\infty,\dots,d_p\infty)' \in
\Rquer^p$. Note that for this choice of $t$, $P_t(\hat u \in M) =
P(\hat u^d \in M)$. Since $d \in \{-1,1\}^p$ was arbitrary, $c \leq
\min_{d \in \{-1,1\}^p} P(\hat u^d \in M) = 1 - \alpha$ follows.
\end{proof}


\begin{proof}[Proof of Proposition~\ref{prop:Vconsist}]
 
Define the function $V_n(u) = n[L_n(\beta_n + \lambda_n^*u/n)
- L_n(\beta_n)]/(\lambda_n^*)^2$ and note that $V_n$ is minimized at
$n(\betaL - \beta_n)/\lambda^*_n$. The function $V_n$ is then
given by
$$
V_n(u) = u'\frac{X'X}{n}u - 2\frac{1}{\lambda^*_n}u'X'\eps + 
2\sum_{j=1}^p \frac{\lambda_{n,j}}{\lambda_n^*} \left[
\Big\vert u_j + \frac{n}{\lambda^*_n}\beta_{n,j}\Big\vert 
- \Big\vert\frac{n}{\lambda^*_n}\beta_{n,j}\Big\vert\right].
$$
Clearly $u'X'Xu/n \to u'Cu$ by assumption. Since $X'\eps/\lambda_n^* =
(n^{1/2}/\lambda_n^*)X'\eps/n^{1/2}$ and $\lambda_n^*/n^{1/2} \to
\infty$ as well as $X'\eps/n^{1/2} = O_P(1)$, the second term in the
above display vanishes in probability. To treat the third term, simply
note that $\lambda_{n,j}/\lambda_n^* \to \lambda_{0,j} \in [0,1]$ and
$n\beta_{n,j}/\lambda_n^* \to \zeta_i \in \Rquer$ by assumption.
Piecing this together yields
$$
V_n(u) \pto u'Cu + 2\sum_{j=1}^p \lambda_{0,j}
\left[\ind_{\{\zeta_j \in \R\}} (|u_j + \zeta_j| - |\zeta_j|) + 
\ind_{\{|\zeta_j| = \infty\}} \sgn(\zeta_j)u_j\right] = V^\zeta(u).
$$
Since $V_n$ and $V^\zeta$ are strictly convex and $V^\zeta$ is non-random,
it follows by \cite{Geyer96TR} that also the corresponding minimizers
converge in probability to the minimizer of the limiting function.
\end{proof}


\begin{proof}[Proof of Proposition~\ref{prop:setscriptM}] 
The equality of the two sets given in the display of
Proposition~\ref{prop:setscriptM} is trivial. We show that the set
$\M$ as defined in (\ref{eq:M}) is equal to the set on the left-hand
side and start by proving that $\M$ is contained in that set. Take any
$m \in \M$, by definition, there exists a $\zeta \in \Rquer^p$ so that
$m$ is the minimizer of $V^\zeta$. We need to show that $|(Cm)_j| \leq
\lambda_{0,j}$ for all $j$. Assume that $|(Cm)_{j_0}| >
\lambda_{0,j_0}$ for some $1 \leq j_0 \leq p$. If $(Cm)_{j_0} >
\lambda_{0,j_0}$ we consider the directional derivative of $V^\zeta$
at its minimizer $m$ in the direction of $-e_{j_0}$ to get
\begin{align*}
\pd{V^\zeta(m)}{(-e_{j_0})} & \; = \; -2(Cm)_j + 2\lambda_{0,j_0} \left[ 
\ind_{\{m_j + \zeta_j \leq 0\}} - \ind_{\{m_j + \zeta_j > 0\}}\right] \\
& \; \leq \; -2(Cm)_j + 2\lambda_{0,j_0} < 0,
\end{align*}
which is a contradiction to $m$ minimizing $V^\zeta$. If $(Cm)_{j_0} <
-\lambda_{0,j_0}$, then consider the directional derivative of
$V^\zeta$ at $m$ in the direction of $e_{j_0}$ to arrive at
\begin{align*}
\pd{V^\zeta(m)}{e_{j_0}} & \; = \; 2(Cm)_j + 2\lambda_{0,j_0} \left[ 
\ind_{\{m_j + \zeta_j \geq 0\}} - \ind_{\{m_j + \zeta_j < 0\}}\right] \\
& \; \leq \; -2(Cm)_j + 2\lambda_{0,j_0} < 0,
\end{align*}
yielding a contradiction also.

To see the reverse set-inclusion, we need to show that for any $m \in
\R^p$ satisfying $|(Cm)_j| \leq \lambda_{0,j}$ for all
$j=1,\dots,p$, there exists a $\zeta \in \Rquer^p$ such that $m$ is
the minimizer of $V^\zeta$. Let $\zeta = -m \in \R^p$ and consider the
directional derivative of $V^\zeta$ at $m$ in any direction $r \in
\R^p\setminus\{0\}$:
$$
\pd{V^\zeta(m)}{r} = 2r'Cm + 2\sum_{j=1}^p \lambda_{0,j} |r_j| \geq 
\sum_{j=1}^p -2|(Cm)_j r_j| + 2\lambda_{0,j}|r_j| =  
2\sum_{j=1}^p \left[-|(Cm)_j| + \lambda_{0,j}\right]|r_j| \geq 0. 
$$
Since the directional derivative is non-negative in any direction $r
\in \R^p\setminus\{0\}$ and $V^\zeta$ is (strictly) convex, $m$ must
be the minimizer.
\end{proof}


\begin{proof}[Proof of Corollary~\ref{cor:coverageconsist}]
We start with the case $d>1$. Let $c = \liminf_{n \to \infty}
\inf_{\beta \in \R^p} P_{\beta}(\beta \in \betaL -
d\lambda_n^*\M/n)$. By definition, there exists a subsequence
$n_k$ and elements $\beta_{n_k} \in \R^p$ such that
$$
P_{\beta_{n_k}}\left(\beta_{n_k} \in \betaL - 
 d\frac{\lambda_{n_k}^*}{n_k}\M\right) = P_{\beta_{n_k}}
\left(\frac{n_k}{\lambda_{n_k}^*}(\betaL - \beta_{n_k}) \in d\M\right) \longto c
$$
as $k \to \infty$. Note that $d\M = \{m \in \R^p : |(Cm)_j| \leq
d\lambda_{0,j},1 \leq j \leq p\}$. Now, pick a further subsequence
$n_{k_l}$ such that $\lambda^*_{n_{k_l}}\beta_{n_{k_l}}/n_{k_l}$
converges in $\Rquer^p$ to, say, $\zeta$.
Proposition~\ref{prop:Vconsist} then shows that $n_{k_l}(\betaL -
\beta_{n_{k_l}})/\lambda^*_{n_{k_l}}$ converges in probability to the
unique minimizer of $V^\zeta$ as $l \to \infty$. Finally,
Proposition~\ref{prop:setscriptM} implies that $c=1$.


We next look the case where $d<1$. Let $m = C^{-1}\lambda_0$ so that
$m \in \M\setminus d\M$. From the proof of
Proposition~\ref{prop:setscriptM}, we know that for $\zeta = -m$ we
have $m = \argmin_{u \in \R^p} V^\zeta(u)$. Let $\beta_n =
n\zeta/\lambda^*_n$. By Proposition~\ref{prop:Vconsist}, $n(\betaL -
\beta_n)/\lambda^*_n$ converges to $m$ in $P_{\beta_n}$-probability,
so that $P_{\beta_n}(n(\betaL - \beta_n)/\lambda^*_n \in d\M) \to 0$.
\end{proof}

\end{appendix}

\section*{Acknowledgements} 

The authors gratefully acknowledge support from DFG grant FOR 916.

\bibliographystyle{ecrev}
\bibliography{journalsFULL,stat,economet}
  
\end{document}